\documentclass[11pt,oneside]{amsart}

\usepackage{amsmath}

\usepackage{amsfonts}

\usepackage{amsthm}

\usepackage{amssymb}

\usepackage[all]{xy}

 \SelectTips{cm}{}
\usepackage[shortlabels]{enumitem}

\usepackage{etex}
\usepackage{verbatim,latexsym,exscale,amssymb,amsmath,amsthm,amsfonts,mathrsfs,microtype,float,pifont,mathtools, stmaryrd}
\usepackage[all]{xy}
\usepackage[usenames, dvipsnames]{color}
\usepackage{tikz-cd}
\usepackage{relsize}
\usepackage{needspace} 
\usepackage{tikz}
\usepackage{amsmath}
\usepackage{url}
\usepackage{algorithm}

\newcommand{\rNum}[1]{\lowercase\expandafter{\romannumeral #1\relax}}
\newdir{ >}{{}*!/-5pt/@{>}}
\everymath{\displaystyle}

\usepackage{tikz}
\usetikzlibrary{shapes,shadows,arrows}
\usetikzlibrary{arrows, decorations.markings}
\usepackage{fontenc}
\usepackage{graphicx}
\usepackage{bm}
\usepackage{listings}
\usepackage{float}
\usepackage{hyperref}
\hypersetup{
	colorlinks=true,
	citecolor=blue,
	linkcolor=blue,
	filecolor=magenta,      
	urlcolor=cyan,
}
\usepackage[capitalize,noabbrev]{cleveref}
\usepackage{colortbl}
\usepackage{subcaption}
\usepackage{caption}
\usepackage{array}
\usepackage{blkarray}
\usepackage{indentfirst}
\usepackage{inputenc}
\usepackage{booktabs}
\usepackage{dsfont}
\usepackage[T1]{fontenc}
\usepackage{textcomp}
\usepackage{lmodern}
\usepackage[english]{babel}
\usepackage{algpseudocode,algorithm,algorithmicx}
\usetikzlibrary{decorations.pathmorphing}
\usepackage[T1]{fontenc}
\usepackage{xfrac}
\usepackage[
textwidth=3cm,
textsize=small,
colorinlistoftodos]
{todonotes}

\tikzstyle{vecArrow} = [thick, decoration={markings,mark=at position
   1 with {\arrow[semithick]{open triangle 60}}},
   double distance=1.4pt, shorten >= 5.5pt,
   preaction = {decorate},
   postaction = {draw,line width=1.4pt, white,shorten >= 4.5pt}]

\tikzstyle{decision} = [diamond, draw, fill=white]
\tikzstyle{line} = [draw, -stealth, thick]
\tikzstyle{elli}=[draw, ellipse, fill=gray!20, minimum height=8mm, text width=5em, text centered]
\tikzstyle{block} = [draw, rectangle, fill=white, text width=8em, text centered, minimum height=15mm, node distance=10em]

\definecolor{Gray}{gray}{0.95}
\graphicspath{{Plots/}}

\makeatletter
\newbox\xrat@below
\newbox\xrat@above
\newcommand{\xrightarrowtail}[2][]{%
  \setbox\xrat@below=\hbox{\ensuremath{\scriptstyle #1}}%
  \setbox\xrat@above=\hbox{\ensuremath{\scriptstyle #2}}%
  \pgfmathsetlengthmacro{\xrat@len}{max(\wd\xrat@below,\wd\xrat@above)+.6em}%
  \mathrel{\tikz [>->,baseline=-.75ex]
                 \draw (0,0) -- node[below=-2pt] {\box\xrat@below}
                                node[above=-2pt] {\box\xrat@above}
                       (\xrat@len,0) ;}}
\makeatother

\makeatletter
\newcommand{\xdbheadrightarrow}[2][]{%
  \ext@arrow 0099\xdbheadfill@{#1}{#2}}%
\newcommand{\xdbheadfill@}{%
  \arrowfill@\relbar\relbar{\mathrel{\vphantom{\rightarrow}\smash{\twoheadrightarrow}}}}
\makeatother

\newcommand{\Ho}{\mathrm{Ho}} 
\newcommand{\C}{\mathcal{C}} 
\newcommand{\D}{\mathcal{D}} 
\newcommand{\V}{\mathcal{V}}
\newcommand{\Sp}{\mathrm{Sp}} 
\newcommand{\Ch}{\mathrm{Ch}_{\mathbb{Z}}} 
\newcommand{\sAb}{\mathrm{sAb}} 
\newcommand{\Mod}{\text{-}\mathrm{mod}}  
\newcommand{\rmod}{\mathrm{mod}\text{-}}
\newcommand{\Alg}{\text{-}\mathrm{alg}}  
\newcommand{\A}{\mathcal{A}} 
\newcommand{\EML}{\mathrm{EML}} 
\newcommand{\s}{\mathbb{S}} 
\newcommand{\Hom}{\mathrm{Hom}}
\newcommand{\bbZ}{\mathbb{Z}} 
\newcommand{\E}{\mathcal{E}} 

\newcommand{\colim}{\mathrm{colim}}
\newcommand{\hend}{\mathrm{hEnd}}
\def\rin{\rotatebox[origin=c]{-90}{$\in$}}
\newcommand{\kos}[2]{{#1}/\!\!/{#2}} 
\newcommand{\tick}{\ding{51}}
\newcommand{\cross}{\ding{56}}
\newcommand{\tw}{\mathrm{Ch}^{(T,N)}(\A)}

 \SelectTips{cm}{}

\title{Levels of algebraicity in stable homotopy theories}

\author{Jocelyne Ishak}
\address[J. Ishak]{Edinburgh, UK}
\email{jocelyne\_ishak@hotmail.com}

\author{Constanze Roitzheim}
\address[C. Roitzheim]{University of Kent \\ School of Mathematics, Statistics and Actuarial Science\\ Sibson building \\ Canterbury, CT2 7FS, UK}
\email{c.roitzheim@kent.ac.uk}

\author{Jordan Williamson}
\address[J. Williamson]{Department of Algebra, Faculty of Mathematics and Physics, Charles University in Prague, Sokolovsk\'{a} 83, 186 75 Praha, Czech Republic}
\email{williamson@karlin.mff.cuni.cz}

\subjclass[2010]{}

\numberwithin{equation}{section}
\theoremstyle{theorem}
\newtheorem{thm}[equation]{Theorem}
\newtheorem{lemma}[equation]{Lemma}
\newtheorem{thm-defi}[equation]{Theorem-Definition}
\newtheorem{prop}[equation]{Proposition}
\newtheorem{cor}[equation]{Corollary}
\newtheorem*{thm*}{Theorem}
\newtheorem*{rigidity thm*}{Rigidity Theorem}
\newtheorem*{main thm*}{Main Theorem}
\newtheorem*{lemma*}{Lemma}
\newtheorem*{cor*}{Corollary}
\newtheorem{conj}[equation]{Conjecture}
\theoremstyle{definition}

\newtheorem{defi}[equation]{Definition}
\newtheorem*{conv*}{Conventions}

\newtheorem{exam}[equation]{Example}

\newtheorem{rem}[equation]{Remark}
\newtheorem*{remarks*}{Remarks}

\theoremstyle{remark}

\newcommand{\changelocaltocdepth}[1]{%
  \addtocontents{toc}{\protect\setcounter{tocdepth}{#1}}%
  \setcounter{tocdepth}{#1}%
}

\makeatletter
\def\l@subsection{\@tocline{2}{0pt}{2.5pc}{5pc}{}}
\makeatother

\makeatletter
\@namedef{subjclassname@2020}{\textup{2020} Mathematics Subject Classification}
\makeatother
\begin{document}
\begin{abstract}
We study several different notions of algebraicity in use in stable homotopy theory and prove implications between them. The relationships between the different meanings of algebraic are unexpectedly subtle, and we illustrate this with several interesting examples arising from chromatic homotopy theory.
\end{abstract}

\maketitle

{\small\tableofcontents}

\section{Introduction}
When studying problems of a homotopical nature, a useful strategy is to reduce to algebra and use insight there to make deductions in topology. Such a reduction often loses a lot of information, but in some cases, the homotopy theory of interest can be completely modelled algebraically. For example, Serre's theorem shows that the homotopy theory of rational spectra is equivalent to the category of graded $\mathbb{Q}$-vector spaces. As such, one might say that the category of rational spectra is algebraic. However, there are many different meanings of algebraic in use in the community. One goal of this project is to clarify the relationship between these different notions. 

In this paper, we study several definitions of algebraicity, give some alternative characterizations of these definitions, and moreover explore the relations between them. Furthermore, we illustrate our results with many examples of interest, arising primarily from chromatic homotopy theory.

\subsection*{Notions of algebraicity} 
Firstly we give an overview of the different definitions and what behaviour they seek to capture. We will discuss some examples later on in this introduction.

Given a stable model category $\C$, there are many levels on which one can measure how algebraic it is. The strictest notion is to require that it is algebraic up to Quillen equivalence; we say that a stable model category $\C$ is \emph{algebraic} if it is Quillen equivalent to a $\Ch$-enriched model category. This ensures that all higher homotopical information (for example, Toda brackets) is determined by algebraic data. If $\C$ has a compact generator, then using Morita theory~\cite{ss_module, dugger_spectral} and machinery from~\cite{ds_enriched}, one can establish that $\C$ is algebraic if and only if $\C$ is Quillen equivalent to modules over a DGA, see~\cref{thm:algebraic}.

This notion of algebraicity is relatively well established but also technically quite strong, 
so in order to get a richer, fuller picture we will weaken this notion to study other kinds of algebraicity. Instead of requiring that all higher homotopical information is determined by algebra, one can ask for only the triangulated structure of the homotopy category $\Ho(\C)$ to be. This amounts to asking that $\Ho(\C)$ is triangulated equivalent to the derived category of modules over a DGA (or a DG-category). If this is the case, we say that $\C$ is \emph{triangulated algebraic}. 

Any stable model category carries a homotopical enrichment in spectra~\cite{lenhardt2012stable}, and therefore associated to any two objects $X,Y \in \C$ there is a homotopy mapping spectrum $R\Hom(X,Y)$. One can ask for this homotopical enrichment to be determined in algebra; this amounts to asking that the action of $\Ho(\Sp)$ on $\Ho(\C)$ factors over the derived category $\mathsf{D}(\bbZ)$, see~\cref{prop:dz}. In this case, we say that $\C$ has a \emph{$\mathsf{D}(\bbZ)$-action.} This furthermore implies that 
$R\Hom(X,Y)$ is an $H\bbZ$-module for all $X,Y \in \C$. 

The action of $\Ho(\Sp)$ on $\Ho(\C)$ also induces an action of the stable homotopy groups of spheres. The less complex this action is, the closer to algebra one might consider $\C$. One can then examine the extremal case when the action of $\pi_{>0}(\mathbb{S})$ is trivial; in this case, we say that $\C$ has a \emph{trivial $\Ho(\Sp)$-action}. 

In the course of studying these notions, we give some alternative characterizations of them. Some of these characterizations have theoretical value in relating the different kinds of algebraicity as we discuss later in this introduction, but some also give valuable criteria for determining how algebraic a given example is. For example, if $\C$ has a single compact generator, we can test if a $\Ho(\Sp)$-action on $\Ho(\C)$ is trivial just on this compact generator, see \cref{prop:generators}. Furthermore, we examine to what extent the triviality of an action can be tested against only a small number of elements in $\pi_{>0}(\mathbb{S})$. We confirm this for a special case in \cref{thm:klocalcohen} and we conjecture that it holds more generally, see \cref{conj:actions}.

Our first main result gives some equivalent characterizations of the aforementioned definitions. For simplicity, in this introduction we only state the results for the category of modules over a ring spectrum, and we refer the reader to the relevant sections in this paper for further details and the more general statements.

\begin{thm*} \leavevmode
Let $R$ be a ring spectrum.
\begin{enumerate}[(i)]
\item (\ref{prop:algebraicEMLHZalg}) The category $R\Mod$ is algebraic if and only if $R$ is weakly equivalent to an $H\bbZ$-algebra as a ring spectrum. 
\item (\ref{prop:dz}) If the category $R\Mod$ has $\mathsf{D}(\bbZ)$-action, the homotopy mapping spectra \linebreak $\Hom_R(M,N)$ are $H\bbZ$-modules for all $M,N \in R\Mod$.
\item (\ref{thm:klocalcohen}) If $R$ is $E(1)$-local at an odd prime $p$, then $R\Mod$ has trivial $\Ho(\Sp)$-action if and only if $R \wedge^L \alpha_1 = 0$ where $\alpha_1 \in \pi_{2p-3}(L_1\mathbb{S})$ denotes the Hopf element.
\item (\ref{lem:nilpotence}) If $R$ is an element of a set of ring spectra which detects nilpotence in the sense of \cref{nilpotence}, then $R\Mod$ has trivial $\Ho(\Sp)$-action. 
\end{enumerate}
\end{thm*}

One might wonder if it is enough to require that the ring $R$ is an $H\bbZ$-module in part (i) to detect algebraicity. We show in \cref{notHZalg} that this condition is not sufficient, by constructing a certain endomorphism ring spectrum which is an $H\bbZ$-module but not an $H\bbZ$-algebra. Moreover, we use the above theorem to study several examples from chromatic homotopy theory such as $K(n)\Mod$ and $MU\Mod$; we refer the reader to \cref{sec:examples} for the full details, and to \cref{table:examples} for a summary of the examples we treat. For now, we give a quick overview of some important motivating examples. 

The simplest (non-trivial) example one might consider is the category of modules over Morava $K$-theory $K(n)$ for some $0 < n < \infty$, see~\cref{sec:kn}. Since $K(n)_*$ is a graded field, the universal coefficient spectral sequence collapses showing that there is a triangulated equivalence $\Ho(K(n)\Mod) \simeq \mathsf{D}(K(n)_*)$. This shows that $K(n)\Mod$ is triangulated algebraic. However, this triangulated equivalence does not preserve higher homotopical information, and one can show that $K(n)\Mod$ is not algebraic, and also that it does not have a $\mathsf{D}(\bbZ)$-action. However, it does have a trivial $\Ho(\Sp)$-action by \cref{lem:nilpotence}. Moreover, we use an obstruction theory in the case of $K(1)$ to illustrate more directly why it does not have a $\mathsf{D}(\bbZ)$-action.

We study several other examples in detail too. In particular, we study the category $L_1\Sp$ of $E(1)$-local spectra and its exotic models, see Examples~\ref{ex:L1Sp},~\ref{ex:fr} and~\ref{ex:E1exotic}. If we work at an odd prime, then Franke~\cite{franke1996uniqueness} as well as Patchkoria and Pstr\k{a}gowski \cite{patchkoriapstragowski} have shown that the category of $E(1)$-local spectra has an exotic model; that is, a category which is triangulated equivalent to $\Ho(L_1\Sp)$ but for which there is no Quillen equivalence to $L_1\Sp$. We show that Franke's exotic model is algebraic in each of the senses described above. However, it is not known whether Franke's exotic model is unique; therefore, we also consider how algebraic a general exotic model for $L_1\Sp$ must be.

We also provide a detailed study of the category of modules over the endomorphism ring spectrum $\E = \Hom_{ku}(H\bbZ, H\bbZ)$ where $ku$ denotes the connective complex $K$-theory spectrum, see \cref{sec:kuHZ}.  We show that the ring spectrum $\E$ is an $H\bbZ$-module, but it is not an $H\bbZ$-algebra, nor is it weakly equivalent to an $H\bbZ$-algebra as a ring spectrum. 
This result is the key to unlocking further counterexamples; for instance, the category of $\E$-modules has a $\mathsf{D}(\bbZ)$-action, but is not algebraic. 

\subsection*{Relating the different notions}
As demonstrated by the examples discussed above, there is a surprising amount of subtlety present in understanding the relations between these different notions of algebraicity. The other main result for this paper is the following theorem which relates the different notions of algebraicity. 

\begin{thm*} Let $\C$ be a stable model category.
\begin{enumerate}[(i)]
\item (\ref{cor:dz} and \ref{ex:ku}) If $\C$ is algebraic, then $\Ho(\C)$ possesses a $\mathsf{D}(\bbZ)$-action. The converse does not hold. 
\item (\ref{cor:dz_trivial} and \ref{exam:moravaK}) If $\Ho(\C)$ has a $\mathsf{D}(\bbZ)$-action, then it also has a trivial $\Ho(\Sp)$-action. The converse does not hold. 
\item (\ref{prop:alg_triang} and \ref{exam:moravaK}) If $\C$ is algebraic, then it is also triangulated algebraic. The converse does not hold. 
\item (\ref{ex:L1Sp}) Being triangulated algebraic does not imply $\mathsf{D}(\bbZ)$-action or trivial $\Ho(\Sp)$-action in general.
\end{enumerate}
\end{thm*}

For simplicity we illustrate the above theorem in the following diagram demonstrating which implications between the different notions hold and which fail. The complexity of this diagram highlights the importance of reconciling the different types of algebraicity which appear in the field.
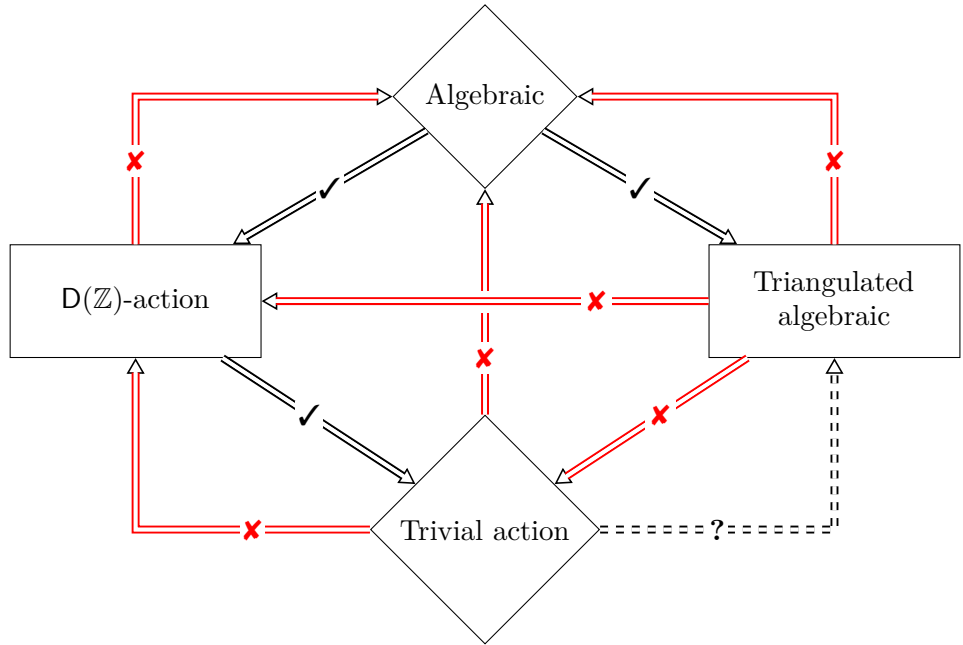
\begin{center}
\begin{tikzpicture}
        \node [decision, yshift=5em] (Algebraic) {Algebraic};
    \node [block, below left of=Algebraic, xshift=-5em] (DZ) {$\mathsf{D}(\mathbb{Z})$-action};
    \node [block, below right of=Algebraic, xshift=5em] (tria) {Triangulated algebraic};
    \node[draw,
    decision,
    below=3cm of Algebraic,
    minimum width=2.5cm,
    minimum height=1cm,] (Trivial action) {Trivial action};
    \draw [vecArrow,red] (DZ) to[]+ (0,1)  |-  node[pos=0.25,fill=white,inner sep=2pt]{\cross} (Algebraic);
    \draw [vecArrow, red] (tria) to[]+ (0,1)  |-  node[pos=0.25,fill=white,inner sep=2pt]{\cross} (Algebraic);

   \draw [vecArrow] (Algebraic) -- (DZ)node[pos=0.5,fill=white,inner sep=1pt]{\tick};
      \draw [vecArrow] (Algebraic) -- (tria)node[pos=0.5,fill=white,inner sep=1pt]{\tick};

        \draw [vecArrow] (DZ) -- (Trivial action)node[pos=0.45,fill=white,inner sep=1pt]{\tick};
   \draw [vecArrow,red] (tria) -- (Trivial action)node[pos=0.45,fill=white,inner sep=2pt]{\cross};

    \draw [vecArrow, red] (Trivial action) -|  node[pos=0.25,fill=white,inner sep=1pt]{\cross} (DZ);
   \draw [vecArrow, dashed]  (Trivial action) -| node[pos=0.25,fill=white,inner sep=1pt]{\textbf{?}} (tria);
      \draw [vecArrow,red] (Trivial action) -- (Algebraic) node[pos=0.5,fill=white,inner sep=4pt]{} node[pos=0.25,fill=white,inner sep=2pt]{\cross};
          \draw [vecArrow,red] (tria) -- (DZ) node[pos=0.25,fill=white,inner sep=2pt]{\cross};
\end{tikzpicture}
\end{center}
In general, it is extremely difficult to determine whether $\Ho(\C)$ is triangulated algebraic. Schwede's notion of $n$-order \cite{schwedealg} can sometimes give a measure to what extent this is \emph{not} the case but remains inconclusive for some examples, see \cref{MU}. In some cases, there is a general machinery to produce a triangulated equivalence to an algebraic model, see \cref{ex:irakli}, but failure to apply this machinery does still not rule out the existence of another triangulated equivalence in general. Therefore, we can only conclude that a characterization of triangulated algebraic model categories is likely to remain an open problem for some time.

\subsection*{Conventions}
A object $X$ of a triangulated category $\mathsf{T}$ is said to be \emph{compact} if $[X,-]^\mathsf{T}$ commutes with arbitrary coproducts. A full triangulated subcategory $\mathcal{L}$ of $\mathsf{T}$ is said to be \emph{localizing} if it is closed under arbitrary coproducts.  An object $X \in \mathsf{T}$ is said to be a \emph{generator} if the smallest localizing subcategory of $\mathsf{T}$ containing $X$ is the whole of $\mathsf{T}$. If $X$ is a compact object in $\mathsf{T}$ then it is a generator if and only if the corepresentable $[X,-]^\mathsf{T}$ detects trivial objects, see for example~\cite[Lemma 2.2.1]{ss_module}. When $\mathsf{T}$ has a compact generator (rather than a set of compact generators) we say that $\mathsf{T}$ is \emph{monogenic}.

We write $\Sp$ to denote a suitable monoidal model category of spectra such as symmetric spectra or orthogonal spectra. We write $\mathbb{S}$ for the sphere spectrum.

\subsection*{Acknowledgements} The authors thank Scott Balchin, David Barnes and Anna Marie Bohmann for helpful discussions. We furthermore thank the referee of the final version for their helpful comments. 
The second author thanks the London Mathematical Society for an Emmy Noether Fellowship.
The third author was supported by the grant GA~\v{C}R 20-02760Y from the Czech Science Foundation. 

\section{Shipley's algebraicization theorem}\label{sec:shipley}
In~\cite{shipley_hz}, Shipley constructs a passage between ring spectra and differential graded algebras. In this section, we give a summary of some of the main results of~\cite{shipley_hz} that will be needed in later sections. Furthermore, we recall the construction of the Eilenberg-Mac Lane spectrum associated to a DGA, which will be crucial for the rest of this paper.

\subsection{Modules, monoids and adjoint lifting}
Given a monoidal model category $\C$, we denote the category of monoid objects in $\C$ by $\mathrm{Ring}(\C)$. Given $S \in \mathrm{Ring}(\C)$, we write $S\Mod(\C)$ for the category of $S$-modules in $\C$. If the underlying category is evident from the context, we will instead write $S\Mod$. Similarly, given a commutative monoid $S$ in $\C$, we write $S\Alg(\C)$ for the category of $S$-algebras in $\C$. Under mild hypotheses, the categories of monoids, modules over a monoid, and algebras over a commutative monoid admit model structures in which the weak equivalences and fibrations are created by the forgetful functors to $\C$, see~\cite[Theorem 4.1]{ss_algebras}.

Suppose we have a weak monoidal Quillen adjunction $L: \D \rightleftarrows \C :R$ 
in the sense of~\cite[Definition 3.6]{ss_equivalences}.
Since the right adjoint $R$ is lax monoidal it preserves monoid objects. Therefore by the adjoint lifting theorem (see also~\cite[\S 3.3]{ss_equivalences}) for every monoid $A$ in $\mathcal{C}$, the Quillen adjunction $(L,R)$ lifts to a Quillen adjunction
 \[\begin{tikzcd}
A\Mod \arrow[r, "R"', yshift=-1mm] & R(A)\Mod \arrow[l, "L^A"', yshift=1mm]
 \end{tikzcd}\]
 between the categories of modules. We note that the functor $L^A$ is different to the underlying functor $L$. When $L$ is strong monoidal, the functor $L^A$ takes the form
 \[L^A(M)=L(M)\wedge_{L(R(A))}A\] for all $M \in R(A)\Mod.$

\subsection{Change of rings}
Recall that given a map of commutative ring spectra $\theta\colon S \to R$ the \emph{restriction of scalars }$\theta^*\colon R\Mod \to S\Mod$ has a left adjoint $\theta_* = R \wedge_S -$ called \emph{extension of scalars.} Together these form a strong monoidal Quillen adjunction, and so given an $R$-algebra $A$ (i.e., a monoid in $R\Mod$), there is an induced Quillen adjunction \[ \theta_*^A: \theta^{\ast}(A)\Mod(S\Mod) \rightleftarrows A\Mod(R\Mod) : \theta^{\ast}\] where $\theta_*^A(X) =\theta_*(X)\wedge_{\theta_*\theta^{\ast}(A)} A$ as described above.
\begin{lemma}\label{rest-mod}
Let $\theta\colon S \to R$ be a map of commutative ring spectra, and let $A$ be an $R$-algebra. The adjunction \[ \theta_*^A: \theta^{\ast}(A)\Mod(S\Mod) \rightleftarrows A\Mod(R\Mod) : \theta^{\ast}\] is a Quillen equivalence.
\end{lemma}
\begin{proof}
Since the restriction of scalars $\theta^*$ reflects weak equivalences, by \cite[Corollary 1.3.16]{hovey2007model}, it is sufficient to show that the derived unit 
$$\eta_M\colon M \to \theta^*\theta_*^AM$$ is an isomorphism for all $M \in \Ho(\theta^*(A)\Mod)$. Consider the full subcategory $\mathcal{L}$ of $\Ho(\theta^{\ast}(A)\Mod)$
consisting of the objects $M$ for which the derived unit $\eta_M$ is an isomorphism. Since $\theta^*$ and $\theta_*^A$ are both exact, coproduct-preserving functors, the subcategory $\mathcal{L}$ is localizing. The subcategory $\mathcal{L}$ also contains $\theta^*(A)$ since by definition of $\theta_*^A$ we have
\[
\theta^{\ast}\theta_*^{A}({\theta^{\ast}(A)})=\theta^{\ast}\left(\theta_*\theta^{\ast}(A)\wedge_{\theta_*\theta^{\ast}(A)} A\right)\simeq \theta^{\ast}(A).\]
Since $\theta^*(A)$ is a generator for $\theta^*(A)\Mod$, the localizing subcategory $\mathcal{L}$ is the whole of $\Ho(\theta^{\ast}(A)\Mod)$ as required.
\end{proof}

\subsection{Eilenberg-Mac~Lane ring spectra associated to a DGA}\label{subsection_EML}
We can associate to any given DGA $\mathcal{A}$ a ring spectrum $\EML(\A)$ called the Eilenberg-Mac~Lane ring spectrum associated to $\A$. In this subsection, we give a brief summary of the construction of this ring spectrum. This is based on \cite{shipley_hz}, and we refer the reader there for more details.

Let $\Ch^{+}$ denote the category of non-negatively graded chain complexes of abelian groups, and $\sAb$ denote the category of simplicial abelian groups. Let $\widetilde{\bbZ}S^1$ be the reduced free simplicial abelian group on the simplicial circle $S^1=\Delta[1]/ \partial \Delta[1]$, and let $\mathbb{Z}[1]$ be the chain complex which contains a single copy of $\mathbb{Z}$ in degree one. We can form the categories $\Sp^{\Sigma}(\Ch^{+},\mathbb{Z}[1])$ and $\Sp^{\Sigma}(\sAb, \widetilde{\bbZ}S^1)$ of symmetric spectra over $\Ch^{+}$ and $\sAb$ as in \cite{hovey2001spectra}. To ease notation we will denote these categories as $\Sp^{\Sigma}(\Ch^{+})$ and $\Sp^{\Sigma}(\sAb)$ respectively.

There are Quillen equivalences~\cite[Proposition 2.10]{shipley_hz}
\begin{equation}\label{eq:modulefunctors}
\begin{tikzcd}
H\bbZ\Mod \arrow[r, "Z", yshift=1mm] & \Sp^\Sigma(\sAb) \arrow[l, "U", yshift=-1mm] \arrow[r, "\phi^*N"', yshift=-1mm] & \Sp^\Sigma(\Ch^+) \arrow[l, "L"', yshift=1mm] \arrow[r, "D", yshift=1mm] & \Ch \arrow[l, "R", yshift=-1mm] 
\end{tikzcd}
\end{equation}
which we briefly describe now. The functor $U$ is induced by the forgetful functor from simplicial abelian groups to simplicial sets and has a left adjoint $Z$. The pair $(Z,U)$ is a strong monoidal Quillen equivalence. The adjunction $(L, \phi^*N)$ is a weak monoidal Quillen equivalence and is a stabilized version of the Dold-Kan correspondence. The right adjoint is given by first applying the normalization functor $$N\colon \sAb \to \Ch^+$$ levelwise, and then restricting scalars along the ring map $$\phi\colon \mathrm{Sym}_{\Ch^+}(\mathbb{Z}[1]) \to \mathcal{N}$$ where $\mathcal{N}=N(\mathrm{Sym}_{\sAb}(\widetilde{\bbZ}S^1))$ and $\mathrm{Sym}_\C\colon \C \to \mathrm{Fun}(\Sigma, \C)$ denotes the free commutative monoid in the category of symmetric sequences. The functor $$R: \Ch \rightarrow \Sp^{\Sigma}(\Ch^{+})$$ is defined by setting $(RY)_m=C_0(Y \otimes \bbZ[m] )$, where $C_0$ is the connective cover, and $\mathbb{Z}[m]$ is the chain complex with a single copy of $\mathbb{Z}$ in degree $m$. This has a left adjoint $D$, and the pair $(D,R)$ forms a strong monoidal Quillen equivalence.

By taking monoid objects in the Quillen equivalences of~(\ref{eq:modulefunctors}) and applying~\cite[Theorem 3.12]{ss_equivalences} one obtains Quillen equivalences
\begin{equation}\label{eq:algebrafunctors}
\begin{tikzcd}
H\bbZ\Alg \arrow[r, "Z", yshift=1mm] & \mathrm{Ring}(\Sp^\Sigma(\sAb)) \arrow[l, "U", yshift=-1mm] \arrow[r, "\phi^*N"', yshift=-1mm] & \mathrm{Ring}(\Sp^\Sigma(\Ch^+))\arrow[l, "L^\mathrm{mon}"', yshift=1mm] \arrow[r, "D", yshift=1mm] & \mathrm{DGA}_\bbZ. \arrow[l, "R", yshift=-1mm] 
\end{tikzcd}
\end{equation}
Taking composites of derived functors in~(\ref{eq:algebrafunctors}) we obtain functors 
\[\mathbb{H}\colon \mathrm{DGA}_\bbZ \to H\bbZ\Alg \quad \mathrm{and} \quad \Theta\colon H\bbZ\Alg \to \mathrm{DGA}_\bbZ\]
defined by $\mathbb{H} = UL^\mathrm{mon}cR$ and $\Theta = Dc\phi^*NZc$, where $c$ denotes cofibrant replacement in the appropriate category of monoids. Note that no fibrant replacements are necessary since each of the right adjoints preserves all weak equivalences. 

We now have the necessary background to recall the following important theorem of Shipley.
\begin{thm}[{\cite[Corollary 2.15]{shipley_hz}}]\label{thm:shipley}\leavevmode
\begin{itemize}
\item[(i)] Let $R$ be an $H\bbZ$-algebra. Then $R\Mod(H\bbZ\Mod) \simeq_Q \Theta(R)\Mod(\Ch)$.
\item[(ii)] Let $\A$ be a DGA. Then $\A\Mod(\Ch) \simeq_Q \mathbb{H}\A\Mod(H\bbZ\Mod).$
\end{itemize}
\end{thm}

\begin{defi} The Eilenberg-Mac Lane ring spectrum associated to a DGA $\A$ is defined as $\EML (\A):= \theta^{\ast} (\mathbb{H}\A)$ where $\theta^{\ast}$ denotes the restriction of scalars along the unit map $\theta\colon \s \to H\bbZ$.
\end{defi}

In particular, if $\A$ is concentrated in degree 0, the spectrum $\EML(\A)$ is the ``classical'' Eilenberg-Mac Lane spectrum $H\A$ whose homotopy groups are 
\[
\pi_0(H\A)\cong \A\,\,\,\mbox{and}\,\,\,\pi_i(H\A)=0\,\,\,\mbox{for}\,\,\,i \neq 0.
\]

\begin{prop}\label{prop:Amod} Let $\theta\colon \mathbb{S} \to H\bbZ$ be the unit map, and $\theta^*$ denote the restriction of scalars functor along $\theta$.
\begin{itemize}
\item[(i)] Let $R$ be an $H\bbZ$-algebra. Then $(\theta^*R)\Mod(\Sp) \simeq_Q \Theta(R)\Mod(\Ch)$.
\item[(ii)] Let $\A$ be a DGA. Then $\A\Mod(\Ch) \simeq_Q \EML(\A)\Mod(\Sp).$
\end{itemize}
\end{prop}
\begin{proof}
For part (i), by \cref{thm:shipley} we have $R\Mod(H\bbZ\Mod) \simeq_Q \Theta(R)\Mod(\Ch)$, and by \cref{rest-mod} we have $R\Mod(H\bbZ\Mod) \simeq_Q (\theta^*R)\Mod(\Sp)$. The proof of (ii) is similar.
\end{proof}

We end this section by recalling the relationship between $H\bbZ$-modules and generalized Eilenberg-Mac Lane spectra in the sense of~\cite{CG}.
\begin{defi}\label{defi:GEML}
A spectrum $X$ is a \emph{generalized Eilenberg-Mac Lane spectrum} if $X$ is weakly equivalent (as a spectrum) to $\vee_{i \in \bbZ} \Sigma^i HA_i$ where each $A_i$ is an abelian group. 
\end{defi}

\begin{prop}[{\cite[Proposition 5.3]{CG}}]\label{prop:CG}
A spectrum is a generalized Eilenberg-Mac~Lane spectrum if and only if it is weakly equivalent to an $H\bbZ$-module.
\end{prop}

Using the previous characterization of generalized Eilenberg-Mac Lane spectra one obtains the following result.
\begin{prop}\label{prop:GEML}
Let $\A$ be a DGA. The spectrum $\mathbb{H}\A$ is a generalized Eilenberg-Mac Lane spectrum. In other words, there is an equivalence of underlying spectra \[\mathbb{H}\A \simeq \bigvee_{i \in \bbZ} \Sigma^i HA_i\] where $A_i = H_i(\A)$, the $i$th homology group of $\A$. 
\end{prop}
\begin{proof}
The spectrum $\mathbb{H}\A$ is an $H\bbZ$-algebra and hence an $H\bbZ$-module, so the claim follows from \cref{prop:CG}.
\end{proof}

We emphasize that the equivalence in the previous proposition is of underlying spectra, and is \emph{not} an equivalence of ring spectra or $H\bbZ$-algebras.

\section{Algebraic model categories and Morita theory}\label{sec:algebraic}
In this section we give a definition of algebraic model categories, and then use Morita theory to give some key characterizations of algebraicity which we will require. We start by recalling the definition of algebraic model categories, and then recall some tools from Morita theory~\cite{ss_module, ds_enriched}. We then use these tools to provide several characterizations of algebraic model categories.

\subsection{Algebraicity}
Recall that given a monoidal model category $\V$, another model category $\C$ is said to be a \emph{$\V$-enriched model category} if the underlying category of $\C$ is enriched, tensored and cotensored over $\V$, and the model structures on $\C$ and $\V$ are suitably compatible, see~\cite[\S 4.3]{gm_enriched} for instance. We write $\C_\V(-,-)$ for the object of $\V$ which gives the enrichment of $\C$ over $\V$.
\begin{defi}\label{defn:algebraic}
A stable model category $\C$ is said to be \emph{algebraic} if it is Quillen equivalent to a combinatorial $\Ch$-enriched model category. 
\end{defi}

\begin{rem}\label{rem:triangulatedequivalence}
It is important to note that any $\Ch$-enriched model category is stable by the analogue of~\cite[Lemma 3.5.2]{ss_module}. As such if $\C$ is algebraic, then the combinatorial $\Ch$-enriched model category $\C^{\mathrm{alg}}$ it is Quillen equivalent to is also stable. Therefore the equivalence $\Ho(\C) \simeq \Ho(\C^{\mathrm{alg}})$ on homotopy categories is necessarily triangulated.
\end{rem}

It will also be convenient to set terminology for a weaker notion of algebraicity.
\begin{defi}
Let $\C$ be a stable model category. We say that $\C$ is \emph{triangulated algebraic} if its homotopy category $\Ho(\C)$ is triangulated equivalent to the stable category of a Frobenius category. If $\Ho(\C)$ is compactly generated, $\C$ is triangulated algebraic if and only if $\Ho(\C)$ is triangulated equivalent to the derived category $\mathsf{D}(\A)$ of a dg-category $\A$ by~\cite[Theorem 3.8]{Keller06}.
\end{defi}

\subsection{Morita theory}
Morita theory is a well-known tool to classify stable model categories with a compact generator in terms of modules over an endomorphism ring spectrum. However, if the stable model category is algebraic in the sense of \cref{defn:algebraic}, then it can be classified using an endomorphism DGA. We start by recalling some key facts about Morita theory from \cite{ss_module, dugger_spectral}, then use the machinery developed in~\cite{ds_enriched} to prove \cref{thm:algebraic} which states that if $\C$ is an algebraic model category then it is Quillen equivalent to $\A \Mod$ for some well-defined DGA $\A$.

A model category $\C$ is said to be \emph{spectral} if it is a $\Sp^\Sigma$-enriched model category, where $\Sp^\Sigma$ denotes the category of symmetric spectra. Recall also that a model category $\C$ is said to be \emph{presentable} if it is Quillen equivalent to a combinatorial model category~\cite[Theorem 4.3]{dugger_spectral}. This is a mild condition on $\C$ which is almost always satisfied in practice. Dugger~\cite[Propositions 5.5 and 5.6]{dugger_spectral} proved that every stable, presentable model category $\C$ is Quillen equivalent to a spectral model category which we will denote by $\C^\mathrm{sp}$. 
\begin{defi}\label{defi:hEnd}
Let $\C$ be a stable, presentable model category, and let $X$ be a bifibrant object in $\C$. We then define a (symmetric) ring spectrum $\mathrm{hEnd}(X)$ by
\[\mathrm{hEnd}(X) = \C^{\mathrm{sp}}_{\Sp^\Sigma}(\overline{X}, \overline{X})\]
where $\overline{X}$ is a bifibrant object of $\C^{\mathrm{sp}}$ which corresponds to $X$ under the Quillen equivalence $\C \simeq_Q \C^\mathrm{sp}$, and $\C^{\mathrm{sp}}_{\Sp^\Sigma}(-,-)$ denotes the enrichment of $\C^{\mathrm{sp}}$ in $\Sp^\Sigma$.
\end{defi}

\begin{thm}[{\cite[Theorem 3.1.1]{ss_module}, \cite[Theorem 8.1]{dugger_spectral}}]\label{thm:morita}
Let $\C$ be a monogenic, stable, presentable model category with a bifibrant compact generator $X$. Then $\C$ is Quillen equivalent to the category of module spectra over the ring spectrum $\hend (X)$. 
\end{thm}

\begin{rem}
As stated, the previous result appears in Dugger~\cite{dugger_spectral}, but the key ideas go back to Schwede and Shipley~\cite{ss_module}. More precisely, Schwede and Shipley~\cite[Theorem 3.8.2]{ss_module} show that if $\C$ is simplicial, cofibrantly generated, proper and stable, then it is Quillen equivalent to a spectral model category $\C^{\mathrm{sp}}$. One may then construct an endomorphism object as in \cref{defi:hEnd}. On the other hand, Dugger~\cite[Propositions 5.5 and 5.6]{dugger_spectral} shows that if $\C$ is stable and presentable then it is Quillen equivalent to a spectral model category. Dugger's assumption of stable and presentable are more appropriate for our purposes since any algebraic model category is by definition presentable.
\end{rem}

Now if $\C$ is algebraic, then using~\cite{ds_enriched} we may construct an endomorphism ring spectrum in simplicial abelian groups denoted $$\mathrm{hEnd}_\mathrm{ad}(X) \in \mathrm{Ring}(\Sp^{\Sigma}(\sAb))$$ for $X \in \C$ as we now recall. If $\C$ is algebraic, it is Quillen equivalent to a combinatorial $\Ch$-model category $\C^{\mathrm{alg}}$. The model category $\C^{\mathrm{alg}}$ is stable, combinatorial and additive by~\cite[Corollary 6.9]{ds_enriched} and therefore is Quillen equivalent to a $\Sp^\Sigma(\sAb)$-model category $\C^\mathrm{ad}$ by~\cite[Theorem 1.3, \S 8.2]{ds_enriched}. 
\begin{defi}\label{defi:hEndad}
Let $\C$ be an algebraic model category and $X$ be a bifibrant object in $\C$. We then define $\mathrm{hEnd}_{\mathrm{ad}}(X) \in \mathrm{Ring}(\Sp^{\Sigma}(\sAb))$ by \[\mathrm{hEnd}_{\mathrm{ad}}(X) = \C^\mathrm{ad}_{\Sp^\Sigma(\sAb)}(\overline{X},\overline{X})\] where $\overline{X}$ is a bifibrant object of $\C^{\mathrm{ad}}$ which corresponds to $X$ under the Quillen equivalence $\C \simeq_Q \C^{\mathrm{ad}}$, and $\C^\mathrm{ad}_{\Sp^\Sigma(\sAb)}(-,-)$ denotes the enrichment of $\C^{\mathrm{ad}}$ in $\Sp^\Sigma(\sAb)$.
\end{defi}

\begin{prop}[{\cite[Proposition 1.5]{ds_enriched}}]\label{prop:hEndad}
Let $\C$ be an algebraic model category and $X \in \C$ be a bifibrant object. Then the endomorphism ring spectrum $\mathrm{hEnd}(X)$ is the Eilenberg-Mac~Lane spectrum associated to $\mathrm{hEnd}_{\mathrm{ad}}(X)$, that is,
\[\theta^*U(\mathrm{hEnd}_\mathrm{ad}(X)) \simeq \mathrm{hEnd}(X).\]
\end{prop}

In \cref{sec:shipley}, we recalled how to construct an Eilenberg-Mac~Lane spectrum associated to a DGA. This construction passes through $\mathrm{Ring}(\Sp^\Sigma(\sAb))$. Therefore, given an object of $\mathrm{Ring}(\Sp^\Sigma(\sAb))$ such as $\mathrm{hEnd}_\mathrm{ad}(X)$ as recalled from~\cite{ds_enriched} in \cref{defi:hEndad} one may produce a DGA as we now describe.

Using the functors described in~(\ref{eq:algebrafunctors}) one can define functors \[\mathbb{H}'\colon \mathrm{DGA}_\bbZ \to \mathrm{Ring}(\Sp^\Sigma(\sAb)) \quad \mathrm{and} \quad \Theta'\colon \mathrm{Ring}(\Sp^\Sigma(\sAb)) \to \mathrm{DGA}_\bbZ\] by $\mathbb{H}' = L^\mathrm{mon}cR$ and $\Theta' = Dc\phi^*N$ where $c$ denotes cofibrant replacement. We note that $\mathbb{H}'\Theta'$ and $\Theta'\mathbb{H}'$ are weakly equivalent to the identity functors since the adjunctions $(L^\mathrm{mon}, \phi^*N)$ and $(D,R)$ are both Quillen equivalences~\cite[Proposition 2.10]{shipley_hz} as described in \cref{sec:shipley}. 
\begin{defi}\label{defn:hEnddga}
Let $\C$ be an algebraic model category, and $X$ be a bifibrant object in $\C$. We may associate a DGA $\mathrm{hEnd}_\mathrm{dga}(X)$ to $X$, defined by \[\mathrm{hEnd}_{\mathrm{dga}}(X) = \Theta'(\mathrm{hEnd}_\mathrm{ad}(X)).\]
\end{defi}

\begin{prop}[{\cite[Proposition 1.7]{ds_enriched}}]\label{ds_algebraic}
Let $\C$ be a combinatorial $\Ch$-model category whose homotopy category is compactly generated. Then for any bifibrant object $X \in \C$, we have $\mathrm{hEnd}_\mathrm{dga}(X) \simeq \C_{\Ch}(X,X)$ where $\C_{\Ch}(-,-)$ denotes the enrichment of $\C$ in $\Ch$.
\end{prop}

We can now give the aforementioned characterization of algebraic model categories. The reader may find it helpful to refer to \cref{fig:ends} for a schematic of the relations between the different endomorphism objects before reading the following proof.
\begin{thm}\label{thm:algebraic}
Let $\mathcal{C}$ be a monogenic, algebraic model category, and write $X$ for a bifibrant compact generator. Then $\mathcal{C}$ is Quillen equivalent to the category of modules over the DGA $\mathrm{hEnd}_\mathrm{dga}(X)$.
\end{thm}
\begin{proof}
We will show that there is a zig-zag of Quillen equivalences given by 
\[\C \simeq_Q \mathrm{hEnd}(X)\Mod \simeq_Q \theta^*\mathbb{H}(\mathrm{hEnd}_\mathrm{dga}(X))\Mod \simeq_Q \mathrm{hEnd}_\mathrm{dga}(X)\Mod.\]
The first Quillen equivalence holds by \cref{thm:morita} and the final Quillen equivalence by \cref{prop:Amod}. Therefore, it remains to justify the second Quillen equivalence. We have
\begin{alignat*}{2}
\theta^*\mathbb{H}(\mathrm{hEnd}_\mathrm{dga}(X)) 
&= \theta^*\mathbb{H}\Theta'(\mathrm{hEnd}_\mathrm{ad}(X)) &\quad &\text{by definition of $\mathrm{hEnd}_\mathrm{dga}(X)$}  \\
&= \theta^*U\mathbb{H}'\Theta'(\mathrm{hEnd}_\mathrm{ad}(X)) &\quad &\text{since $\mathbb{H} = U\mathbb{H}'$ by definition}  \\
&\simeq \theta^*U(\mathrm{hEnd}_\mathrm{ad}(X)) &\quad &\text{as $\Theta'$ and $\mathbb{H}'$ are inverse equivalences} \\
&\simeq \mathrm{hEnd}(X)\Mod &\quad &\text{by \cref{prop:hEndad}}
\end{alignat*}
and hence the second Quillen equivalence holds. 
\end{proof}

By combining the previous theorem with \cref{ds_algebraic} one obtains the following corollary.
\begin{cor}
Let $\C$ be a monogenic, combinatorial, $\Ch$-enriched model category, and write $X$ for a bifibrant compact generator. Then $\C$ is Quillen equivalent to the category of modules over the DGA $\C_{\Ch}(X,X)$. \qed
\end{cor}

\begin{figure}[H]
\centering
\framebox{
\begin{tikzcd}[row sep=0cm, ampersand replacement=\&]
\mathrm{hEnd}(X) \& \mathrm{hEnd}_\mathrm{ad}(X) \arrow[l, mapsto, "\mathrm{forget}"', "{(\ref{prop:hEndad})}"] \arrow[r, mapsto, "\Theta'", "{(\ref{defn:hEnddga})}"'] \& \mathrm{hEnd}_\mathrm{dga}(X) \arrow[rr, leftrightsquigarrow, "\substack{\text{if $\C$ is } \\ \text{$\Ch$-enriched} \\ {(\ref{ds_algebraic}})}"', "\simeq"] \& \& \C_{\Ch}(X,X) \\
\rin \& \rin \& \rin \& \& \rin \\
\mathrm{Ring}(\Sp^\Sigma) \& \mathrm{Ring}(\Sp^\Sigma(\sAb)) \& \mathrm{DGA}_\bbZ \& \& \mathrm{DGA}_\bbZ
\end{tikzcd}}
\caption{The various endomorphism objects associated to a bifibrant object $X$ in an algebraic model category $\C$ and how they relate to one another.}
\label{fig:ends}
\end{figure}

\subsection{Detecting algebraicity}\label{detect_algebraicity}
We can now use the Morita theory results from above to give several criteria for when stable model categories are algebraic. Firstly, we give the following direct consequence of \cref{thm:algebraic}.
\begin{thm}\label{prop:algebraicisDGmod}
Let $\C$ be a monogenic, stable model category. Then $\C$ is algebraic if and only if $\C$ is Quillen equivalent to $\A\Mod$ for a DGA $\A$. \qed
\end{thm}

\begin{prop}\label{prop:alg_triang}
Let $\C$ be a monogenic, stable model category. If $\C$ is algebraic then it is triangulated algebraic. 
\end{prop}
\begin{proof}
By \cref{prop:algebraicisDGmod}, if $\C$ is algebraic then $\C$ is Quillen equivalent to $\A\Mod$ for a DGA $\A$ and the derived equivalence is triangulated since the zig-zag of Quillen equivalences of \cref{thm:algebraic} only passes through stable model categories. Hence, $\C$ is triangulated algebraic. 
\end{proof}

\begin{thm}\label{prop:algebraicEMLHZalg}
Let $\C$ be a monogenic, stable model category and write $X$ for a bifibrant compact generator. Then $\C$ is algebraic if and only if $\mathrm{hEnd}(X)$ is weakly equivalent to an $H\bbZ$-algebra as a ring spectrum.
\end{thm}
\begin{proof}
If $\C$ is algebraic, then \cref{prop:hEndad} shows that $\mathrm{hEnd}(X)$ is weakly equivalent to the $H\bbZ$-algebra $U\mathrm{hEnd}_\mathrm{ad}(X)$
as a ring spectrum. Conversely, if $\mathrm{hEnd}(X)$ is weakly equivalent to an $H\bbZ$-algebra $R$ as a ring spectrum, then \[\C \simeq_Q \mathrm{hEnd}(X)\Mod \simeq_Q R\Mod \simeq_Q \Theta(R)\Mod\] by Theorems~\ref{thm:morita} and~\ref{thm:shipley}. The object $\Theta(R)$ is a DGA and therefore by \cref{prop:algebraicisDGmod} it follows that $\C$ is algebraic.
\end{proof}

In the case when $\C$ is the category of modules over a ring spectrum, we can give more explicit versions of the previous results.
\begin{cor}\label{cor:Rmodalgebraic}
Let $R$ be a ring spectrum. Then $R\Mod$ is algebraic if and only if $R$ is weakly equivalent to an $H\bbZ$-algebra as a ring spectrum. In particular, if $R\Mod$ is algebraic then $R$ is a generalized Eilenberg-Mac~Lane spectrum in the sense of \cref{defi:GEML}.
\end{cor}
\begin{proof}
The category $R\Mod$ is spectral with enrichment in spectra given by $\Hom_R(-,-)$. Therefore the endomorphism object $\mathrm{hEnd}(R)$ associated to the compact generator $R$ is just $R$. The first claim then follows from \cref{prop:algebraicEMLHZalg}. As for the second claim, if $R\Mod$ is algebraic, then the first part of this corollary tells us that $R$ is weakly equivalent to an $H\bbZ$-algebra and hence an $H\bbZ$-module, and the result then follows from \cref{prop:CG}.
\end{proof}

\begin{rem}
 It is reasonable to wonder if in \cref{prop:algebraicEMLHZalg} it is instead enough to require that $\mathrm{hEnd}(X)$ is an $H\bbZ$-module (since it is naturally a ring spectrum). In \cref{notHZalg} we give an example which shows that this is false.
\end{rem}

\section{Detecting algebraicity via $\Ho(\Sp)$-actions}\label{sec:actions}

For any stable model category $\C$ there is an action of $\Ho(\Sp)$ on $\Ho(\C)$. In this section we will show that being algebraic implies that this $\Ho(\Sp)$-action factors over a $\mathsf{D}(\bbZ)$-action. Furthermore, we will consider the notion of a trivial $\Ho(\Sp)$-action on $\Ho(\C)$ and explore how these concepts interact. 

\subsection{Actions of $\Ho(\Sp)$ and $\mathsf{D}(\bbZ)$}
A useful structure to have on a stable model category $\C$ is a tensor with spectra $X \wedge A$ where $X \in \C$ and $A \in \Sp$, and mapping spectra $\Hom(X,Y) \in \Sp$ behaving in a homotopically useful way. This would mean asking for $\C$ to be a ``spectral'' model category, which is a very strong assumption. However, it has been shown by \cite{lenhardt2012stable} that such a structure at least exists up to homotopy.

\begin{thm}[{\cite[Theorem 6.3]{lenhardt2012stable}}]\label{thm:closedaction}
For every stable model category $\C$, there is a bifunctor 
\[
- \wedge^L -\colon \Ho(\C) \times \Ho(\Sp) \longrightarrow \Ho(\C)
\]
making  $\Ho(\C)$ a closed module category over the stable homotopy category $\Ho(\Sp)$ in the sense of~\cite[Definition 4.1.6]{hovey2007model}. For $X \in \C$, we call the right adjoint 
\[
R\Hom(X,-)\colon \Ho(\C) \to \Ho(\Sp)
\]
the \emph{homotopy mapping spectrum functor}.
\end{thm}


This action satisfies the expected properties, such as the following.
\begin{itemize}
\item The functor $- \wedge^L -$ is exact in both variables.
\item We have $X \wedge^L \mathbb{S} \cong X$ for all $X \in \Ho(\C)$.
\item A Quillen functor $\C \longrightarrow \D$ induces a $\Ho(\Sp)$-module functor $\Ho(\C) \longrightarrow \Ho(\D)$.
\item In particular, a Quillen equivalence induces an equivalence of $\Ho(\Sp)$-modules.
\end{itemize}

\begin{rem}
If $\C$ is itself a spectral model category, then the module structure from above theorem agrees with the action derived from the spectral structure. This means that if $\C$ is spectral and presentable, then $\mathrm{hEnd}(X)$ from \cref{defi:hEnd} is weakly equivalent to $R\Hom(X,X).$ 
\end{rem}

In light of the relationship between algebraicity and $H\bbZ$-algebras explored in the previous section, we now explore when the action of $\Ho(\Sp)$ factors over $\mathsf{D}(\bbZ)$.

\begin{prop}\label{prop:dz}
Let $\C$ be a stable model category. If the $\Ho(\Sp)$-module structure from \cref{thm:closedaction} factors over $\mathsf{D}(\bbZ)$, then the homotopy mapping spectra $R\Hom(X,Y)$ are $H\bbZ$-modules for all $X,Y \in \C$, and hence are generalized Eilenberg-Mac~Lane spectra in the sense of \cref{defi:GEML}.
\end{prop}
\begin{proof}
Firstly, we know by~\cite[Proposition 2.10]{shipley_hz} that $\Ch$ is monoidally Quillen equivalent to $H\mathbb{Z}\Mod$. Thus, having $\mathsf{D}(\bbZ)$-action is equivalent to the $\Ho(\Sp)$-action on $\C$ factoring over $\Ho(H\mathbb{Z}\Mod)$ as follows. Consider the diagram
\[
\xymatrix{ \Ho(\C) \times \Ho(\Sp) \ar[rr]^-{-\wedge^L -} \ar[d]_{\Ho(\C)\times (-\wedge H\mathbb{Z})} & & \Ho(\C) \\
\Ho(\C) \times  \Ho(H\mathbb{Z}\Mod) \ar[urr]_{-\wedge^L -} & & 
}
\]

By taking right adjoints, this is equivalent to having the following commutative diagram
\[
\xymatrix{ \Ho(\C)^{\mathrm{op}} \times \Ho(\C) \ar[d]_{R\Hom(-,-)} \ar[rr]^-{R\Hom(-,-)} & & \Ho(\Sp) \\
\Ho(H\mathbb{Z}\Mod) \ar[urr] \\
}
\]
where the arrow $\Ho(H\mathbb{Z}\Mod) \to \Ho(\Sp)$ is the forgetful functor. Thus, we can see that if the $\Ho(\Sp)$-action factors over a $\Ho(H\mathbb{Z}\Mod)$-action, the homotopy mapping spectra $R\Hom(X,Y)$ are $H\mathbb{Z}$-modules for all $X, Y$ in $\C$. The final claim holds by applying \cref{prop:CG}.
\end{proof}

\begin{rem}
Conversely, by unravelling the definition, to obtain a $\mathsf{D}(\bbZ)$-action one needs the following.
\begin{itemize}
\item The homotopy mapping spectra $R\Hom(X,Y)$ are $H\bbZ$-modules for all $X$ and $Y$.
\item The composition map $R\Hom(Y,Z) \wedge^L_{H\bbZ} R\Hom(X,Y) \to R\Hom(X,Z)$ is a map of $H\bbZ$-modules.
\item The unit map $\mathbb{S} \to R\Hom(X,X)$ factors over $H\bbZ$.
\item For $f\colon A \to B$ in $\C$, the induced maps of spectra $R\Hom(X,f)$ and $R\Hom(f,Y)$ are $H\bbZ$-module maps.
\end{itemize}
\end{rem}

We now show that algebraic model categories have a $\mathsf{D}(\bbZ)$-action.
\begin{prop}\label{cor:dz}
Let $\C$ be an algebraic model category. Then the $\Ho(\Sp)$-module structure factors over $\mathsf{D}(\bbZ)$.
\end{prop} 
\begin{proof}
Without loss of generality, we may assume that $\C$ is a $\Ch$-model category rather than just Quillen equivalent to one, as Quillen equivalences induce isomorphic $\Ho(\Sp)$-module structures on the respective homotopy categories. As $\C$ is a $\Ch$-model category by assumption, $\Ho(\C)$ is a closed $\mathsf{D}(\bbZ)$-module, so the claim follows.
\end{proof}

\begin{cor}
A Quillen functor $F\colon\C \longrightarrow \D$ between algebraic model categories induces a functor $\Ho(\C) \longrightarrow \Ho(\D)$ of closed $\mathsf{D}(\bbZ)$-module categories, i.e., the diagram 
\[
\xymatrix{ \Ho(\C) \times \Ho(\Sp) \ar[rr]^{LF \times \Ho(\Sp)} \ar[d] \ar@/_4pc/[dd]_{\wedge^L} & & \Ho(\D) \times \Ho(\Sp) \ar[d] \ar@/^4pc/[dd]^{\wedge^L} \\
\Ho(\C) \times \mathsf{D}(\bbZ) \ar[rr]^{LF \times \mathsf{D}(\bbZ)} \ar[d]^{\wedge^L} & & \Ho(\D) \times \mathsf{D}(\bbZ) \ar[d]_{\wedge^L} \\
\Ho(\C) \ar[rr]_{LF} & & \Ho(\D) 
}
\]
commutes. \qed
\end{cor}

\subsection{Trivial actions}
The action of the stable homotopy category induces an action of the stable homotopy groups $\pi_*(\mathbb{S})$ on the morphism groups in $\Ho(\C)$
\[
 \pi_*(\mathbb{S}) \otimes [X,Y]^{\C}_*  \longrightarrow [X,Y]^{\C}_*
\]
given by $\alpha \otimes \varphi = \varphi \wedge^L \alpha$. Note that the stable homotopy groups act from the left, see \cite[Construction 2.4]{schwedeshipleyuniqueness}, and that the action is associative and unital. We can describe this action explicitly using mapping spectra. 

\begin{lemma}\label{lem:endoaction}
Let $\varphi \in [X,Y]_k^\C$, and $\alpha \in \pi_i(\mathbb{S})$, $i\geq 0$. Then $\alpha \otimes \varphi \in [X,Y]^\C_{i+k}$ is adjoint to the element
\[
\alpha \circ f \in \pi_{i+k}(R\Hom(X,Y)),
\]
where $f \in \pi_k(R\Hom(X,Y))$ is adjoint to $\varphi$.
\end{lemma}

\begin{proof}

Write $\wedge$ instead of $\wedge^L$.
By adjunction, we have that 
\[
\pi_k(R\Hom(X,Y))\cong[\Sigma^k X, Y]^\C,
\]
so $\varphi \in [\Sigma^k X, Y]^\C$ then corresponds to an element $f \in \pi_k(R\Hom(X,Y))$. 
Thus,
\[
\varphi \wedge \alpha \in [\Sigma^{i+k} X, Y]^\C.
\]
By adjunction, this element $\varphi \wedge \alpha$ corresponds to an element $$\alpha \otimes f \in \pi_{i+k}(R\Hom(X,Y)).$$ Let us now specify what this element $\alpha \otimes f$ is.
The adjunction isomorphism gives us
\[
\varphi = \epsilon \circ (X \wedge f)\colon \Sigma^k X \xrightarrow{X \wedge f} X \wedge R\Hom(X,Y) \xrightarrow{\epsilon} Y,
\]
where $\epsilon$ is the counit. The element $\alpha$ now acts on $\varphi$ as
\begin{eqnarray}
\varphi \wedge \alpha & = & (\epsilon \circ (X \wedge f)) \wedge \alpha \nonumber\\
 & = & \epsilon \circ ((X \wedge f) \wedge \alpha) \nonumber \\
 & = & \epsilon \circ ( X \wedge (\alpha \circ f)). \nonumber 
\end{eqnarray}
The second equality comes from the fact that the action is central in the sense that
\[
(\varphi_2 \circ \varphi_1) \wedge \alpha = (\varphi_2 \wedge \alpha) \circ \varphi_1 = \varphi_2 \circ (\varphi_1 \wedge \alpha).
\]
Furthermore, the last equality uses the associativity of the action $\wedge$, however, we have to let $\alpha$ act on $f$ from the left rather than the right in the last step. 

By the same construction we see that $\epsilon \circ ( X \wedge (\alpha \circ f))$ is also adjoint to $\alpha \circ f$. Therefore, we can conclude that $$\alpha \circ f = \alpha \otimes f$$ as claimed.
\end{proof}

\begin{cor}\label{cor:emlaction}
Let $\alpha \in \pi_i(\mathbb{S})$ for $i > 0$. If the mapping spectrum $R\Hom(X,Y)$ is a generalised Eilenberg-Mac Lane spectrum in the sense of \cref{defi:GEML}, then $\alpha \otimes f=0$ for $f \in [X,Y]^\C_*.$
\end{cor}

\begin{proof}
If $E$ is an Eilenberg-Mac Lane spectrum (i.e., a spectrum with homotopy groups concentrated in one degree), then $\alpha \circ f=0$ for any $\alpha \in \pi_i(\mathbb{S})$, $i>0$ and $f \in \pi_*(E)$ because of degree reasons. Therefore, the same is true if $E$ is a wedge of suspensions of Eilenberg-Mac Lane spectra. 
\end{proof}

Given a non-zero $f\in [X,Y]^{\C}$ and $\alpha \in \pi_i(\mathbb{S})$ one sees that $\alpha \otimes f = (\alpha \otimes \mathrm{id}_X) \circ (\mathrm{id}_\mathbb{S} \otimes f)$. Therefore $\alpha \otimes f = 0$ if and only if $X \wedge^L \alpha = \alpha \otimes \mathrm{id}_X = 0$. As such, the previous corollary suggests the following definition.
\begin{defi}\label{def:trivialaction}
We say that $\Ho(\Sp)$ \emph{acts trivially} on $\Ho(\C)$ if for all $X \in \Ho(\C)$ and $\alpha \in \pi_i(\mathbb{S}), \,\,i>0$ we have $X \wedge^L \alpha =0.$
\end{defi}

We now record how trivial action relates to a $\mathsf{D}(\bbZ)$-action and to being algebraic.
\begin{cor}\label{cor:dz_trivial}
If the $\Ho(\Sp)$-action on $\Ho(\C)$ factors over a $\mathsf{D}(\bbZ)$-action, then $\Ho(\Sp)$ acts trivially on $\Ho(\C)$.
In particular, if $\C$ is an algebraic model category, then $\Ho(\Sp)$ acts trivially on $\Ho(\C)$.
\end{cor}
\begin{proof}
By \cref{lem:endoaction}, the action of $\pi_*(\mathbb{S})$ on mapping objects is given by precomposition with the adjoint maps. If the $\Ho(\Sp)$-action factors over $\mathsf{D}(\bbZ)$, this means that the mapping spectra are products of Eilenberg-Mac Lane spectra by \cref{prop:dz}, and thus by \cref{cor:emlaction} the action is trivial. If $\C$ is algebraic, then it has $\mathsf{D}(\bbZ)$-action by \cref{cor:dz} so the second claim follows.
\end{proof}

We can reduce the definition of trivial action to testing on a compact generator.

\begin{prop}\label{prop:generators}
Let $\C$ be a monogenic, stable, presentable model category with a bifibrant compact generator $G$. Then $\Ho(\Sp)$ acts trivially on $\Ho(\C)$ if and only if $G \wedge^L \alpha = 0$ for all $\alpha \in \pi_k(\mathbb{S}),\,\, k>0$.
\end{prop}

\begin{proof}
Since $\C$ is stable and presentable, it is Quillen equivalent to a spectral model category by~\cite[Propositions 5.5 and 5.6]{dugger_spectral}. As Quillen equivalences induce equivalent $\Ho(\Sp)$-module structures, we may without loss of generality assume that $\C$ is spectral.

By Morita theory (c.f.,~\cref{thm:morita}) we have a Quillen equivalence
\[\begin{tikzcd}
\C \arrow[rrr, yshift=-1mm, "{\C_{\Sp^\Sigma}(G,-)}"'] & & & \rmod\mathrm{hEnd}(G) \arrow[lll, yshift=1mm, "- \wedge_{\mathrm{hEnd}(G)} G"'].
\end{tikzcd} \]

Now let $X \in \C$. Since Quillen equivalences induce equivalent $\Ho(\Sp)$-module structures, we have $X \wedge^L \alpha = 0$ if and only if $\C_{\Sp^\Sigma}(G,X) \wedge^L \alpha = 0$.

Since $G \wedge^L \alpha = 0$ by assumption, we also have \[\mathrm{hEnd}(G) \wedge^L \alpha = \C_{\Sp^\Sigma}(G,G) \wedge^L \alpha = 0.\]
But $\C_{\Sp^\Sigma}(G,X) \wedge \alpha$ is a retract of $\C_{\Sp^\Sigma}(G,X) \wedge \C_{\Sp^\Sigma}(G,G) \wedge \alpha$, as the spectrum $\C_{\Sp^\Sigma}(G,X)$ is a module over $\C_{\Sp^\Sigma}(G,G)$.
This implies that $\alpha$ also acts trivially on $\C_{\Sp^\Sigma}(G,X)$ and therefore on $X$, as required. 
\end{proof}

As well as reducing to checking trivial actions on a compact generator of $\Ho(\C)$, it is natural to ask whether or not we can reduce to checking just a small number of elements of $\pi_*(\mathbb{S})$. By~\cite{cohen}, the whole of $\pi_*(\mathbb{S})$ is ``generated'' by the elements of Adams filtration 1 using multiplication and higher Toda brackets: every $\theta \in \pi_*(\mathbb{S})$ that is not a Hopf element can be written as a higher-order Toda bracket of matrices with values in (smaller degrees of) $\pi_*(\mathbb{S})$.
Considering localization at each prime separately, for $p=2$ the elements of Adams filtration 1 are the Hopf elements $2 \in \pi_0(\mathbb{S})$, $\eta \in \pi_1(\mathbb{S})$, $\nu\in\pi_3(\mathbb{S})$ and $\sigma \in \pi_7(\mathbb{S})$. For odd primes, the only ones are $p \in \pi_0(\mathbb{S})$, $\alpha_1 \in \pi_{2p-3}(\mathbb{S})$. 

This suggests the following definition.
\begin{defi}
We say that $\Ho(\Sp)$ \emph{acts essentially trivially} on $\Ho(\C)$ if for all $X \in \Ho(\C)$ and $\alpha \in \pi_i(\mathbb{S}), \,\,i>0,$ of Adams filtration 1 we have $X \wedge^L \alpha =0.$ By the analogue of \cref{prop:generators}, if $\C$ is monogenic, this is equivalent to $G \wedge^L \alpha = 0$ for a compact generator $G$, and $\alpha \in \pi_i(\mathbb{S}),\,\,i>0,$ of Adams filtration 1.
\end{defi}

We make the following conjecture relating essentially trivial actions with trivial actions. We prove a special case of this as \cref{thm:klocalcohen}.
\begin{conj}\label{conj:actions}
Let $\C$ be a stable model category. Then $\Ho(\Sp)$ acts trivially on $\Ho(\C)$ if and only if $\Ho(\Sp)$ acts essentially trivially on $\Ho(\C)$.
\end{conj}

\begin{rem}
One might think that the above conjecture can be proved using Cohen's theorem. However, this method fails mainly because of how the indeterminacy of $n$-fold Toda brackets behave. More precisely, Cohen's theorem~ \cite[Theorem 4.2]{cohen} says that if $\theta \in \pi_*(\mathbb{S})$ is not a Hopf element, then there is a Toda bracket of the form $$\left< f_1, f_2, \cdots f_n \right>$$ containing $\theta$,  where $$G_1 \xrightarrow{f_1} G_2 \xrightarrow{f_2} \cdots \xrightarrow{f_n} G_{n+1},$$ $G_1$ and $G_{n+1}$ are spheres and the other $G_i$ are wedges of spheres. In particular, the $f_i$ can be represented as matrices with entries in $\pi_*(\mathbb{S})$ of degree strictly less than that of $\theta$. We then have that $$X \wedge^L \theta \in \left< X \wedge^L f_1, X \wedge^L f_2, \cdots X \wedge^L f_n \right>.$$ 

It would then be the goal to prove that $X \wedge^L \theta =0$ if $X \wedge^L \alpha=0$ for all Hopf elements $\alpha$ by induction on degree.
At $p=2$, assuming that $X \wedge^L \eta$, $X \wedge^L \nu$ and $X \wedge^L \sigma$ are all trivial, one can quickly prove that $X \wedge^L \theta=0$ for all $\theta$ in degrees 1 to 7. (At odd primes, the situation is even simpler.) Thus, induction tells us that $X \wedge^L \theta$ is an element of a Toda bracket whose entries are zeros and powers of $p$. As such, this Toda bracket contains zero. Sadly, this is not enough to conclude that $X \wedge^L \theta$ itself has to be zero, as a long bracket as above can still have nontrivial indeterminacy. (For example, the bracket $\left<2,0,0,2\right>$ in $\Ho(\Sp)$ contains $\eta^2$.)
\end{rem}

 \subsection{Detecting trivial actions}
In this section we give three criteria for checking whether $\Ho(\Sp)$ acts trivially on $\Ho(\C)$. 
Firstly, in some cases we can detect trivial $\Ho(\Sp)$-action already from homotopy groups.
\begin{lemma}\label{ringunit}
Let $R$ be a ring spectrum and $\iota\colon \mathbb{S} \longrightarrow R$ its unit map. If $\iota_*\colon \pi_*(\mathbb{S}) \longrightarrow \pi_*(R)$ is zero in positive degrees, then $\Ho(R\Mod)$ has trivial $\Ho(\Sp)$-action.
\end{lemma}

\begin{proof}
Let $\mu\colon R \wedge^L R \longrightarrow R$ be the multiplication map of $R$. By \cref{prop:generators}, we want to check that $R \wedge^L \alpha=0$ for $\alpha \in \pi_k(\mathbb{S})$, $k>0$. We have 
$$\mathrm{id} \simeq \mu \circ \iota\colon R \cong R \wedge^L \mathbb{S} \longrightarrow R \wedge^L R \longrightarrow R$$
by definition. Therefore, for $\alpha \in \pi_k(\mathbb{S})$, $k >0$,
$R \wedge^L \alpha \simeq R \wedge^L (\mu \circ \iota \circ \alpha) \simeq R \wedge^L (\mu \circ \iota_*(\alpha)),$
which is zero by assumption.
\end{proof}

\begin{cor}
Let $R$ be a ring spectrum with $\pi_k(R)$ torsion-free for $k>0$. Then $\Ho(R\Mod)$ has trivial $\Ho(\Sp)$-action. \qed
\end{cor}

\begin{cor}\label{ringcoefficients}
For $R= MU,~BP,~E(n),~KU$ and their connective covers, $\Ho(R\Mod)$ has trivial $\Ho(\Sp)$-action. \qed
\end{cor}

Our second criterion gives a positive answer to \cref{conj:actions} in the case when the homotopy mapping spectra are $E(1)$-local at an odd prime.
\begin{thm}\label{thm:klocalcohen}
Let $\C$ be a stable model category for which the homotopy mapping spectra $R\mathrm{Hom}(X,Y)$ are $E(1)$-local for all $X$ and $Y$ in $\C$, for some fixed odd prime $p$. If \linebreak $\alpha_1 \in \pi_{2p-3}(L_1\mathbb{S})$ acts trivially, then $\Ho(\Sp)$ acts trivially on $\Ho(\C)$.
\end{thm}
\begin{proof}
If the homotopy mapping spectra are all $E(1)$-local, then the action of $\Ho(\Sp)$ on $\Ho(\C)$ factors over the $E(1)$-local stable homotopy category $\Ho(L_1\Sp)$~\cite[Theorem 7.8]{barnes_roitzheim_local}, and we write $[A,B]^{L_1\Sp}$ for morphisms in $\Ho(L_1\Sp)$. Therefore it is sufficient to show that $\Ho(L_1\Sp)$ acts trivially on $\Ho(\C)$, because the action of the stable homotopy groups factors as 
\[
\xymatrix{ \pi_{>0}(\mathbb{S}) \otimes [X,Y]^\C \ar[d] \ar[rr] & & [X,Y]^\C. \\
\pi_{>0}(L_1\mathbb{S})  \otimes [X,Y]^\C \ar[urr]& & 
}
\]

By assumption, $X \wedge^L \theta = 0$ for all $\theta \in \pi_i(L_1\mathbb{S})$, $i=1,...,4p-6$ as the only nontrivial homotopy in this range is the group $\pi_{2p-3}(L_1\mathbb{S})=\mathbb{Z}/p$, which is generated by $\alpha_1$. We have short exact sequences of the form
\[
0 \longrightarrow \pi_{i+1}(L_1\mathbb{S})/(p) \longrightarrow [M,L_1\mathbb{S}]_i = [M,\mathbb{S}]_i^{L_1\Sp} \longrightarrow \Gamma_p(\pi_i(L_1\mathbb{S})) \longrightarrow 0,
\]
where $M$ denotes the mod $p$ Moore spectrum and $\Gamma_pG$ the $p$-torsion of an abelian group $G$.  For degree reasons these short exact sequences take the form
\[
0 \longrightarrow 0 \longrightarrow G \longrightarrow G \longrightarrow 0
\quad
\mathrm{or} 
\quad
0 \longrightarrow G \longrightarrow G \longrightarrow 0 \longrightarrow 0
\]
in low degrees.
Therefore, we can conclude that $[M,\mathbb{S}]^{L_1\Sp}_i=0$ in degrees $0 \leq i \leq 2p-5$. In addition, $[M,\mathbb{S}]^{L_1\Sp}_{2p-4}$ is generated by $\alpha_1 \circ  \mathrm{pinch}$, and  $[M,\mathbb{S}]^{L_1\Sp}_{2p-3}$ is generated by an element $A$ with $A \circ \mathrm{incl} = \alpha_1$.


We have a $v_1$-self map $v_1\colon \Sigma^{2p-2}M \to M$ which is an isomorphism in $E(1)$-homology as $M$ is rationally trivial. Therefore, there is a commutative diagram
\[
\begin{tikzcd}
\left [M, \mathbb{S}\right ]_i^{L_1\Sp} \arrow[rr, "X \wedge^L -"] \arrow[d, "v_1^*"'] \arrow[d, "\cong"] & & \left [X \wedge M, X\right ]_i \arrow[d, "(X\wedge^L v_1)^*"]^{L_1\Sp} \arrow[d, "\cong"'] \\
\left [M, \mathbb{S}\right ]_{i+2p-2}^{L_1\Sp} \arrow[rr, "X \wedge^L -"'] & & \left [X \wedge M, X\right ]_{i+2p-2}^{L_1\Sp}.
\end{tikzcd}\]
We have that $[M,\mathbb{S}]^{L_1\Sp}_i=0$ in degrees $0 \leq i \leq 2p-5$, therefore the top arrow is zero in those degrees. 
In degree $2p-4$, we have that $$X \wedge^L (\alpha_1 \circ  \mathrm{pinch})=(X \wedge^L \alpha_1) \circ (X \wedge^L \mathrm{pinch})= 0 \circ (X \wedge^L \mathrm{pinch})=0$$ by assumption.
In degree $2p-3$, we have $X \wedge^L A =0$ as by our previous calculation, $X \wedge^L A$ is zero if and only if $X \wedge^L \alpha_1$ is, because precomposition with the inclusion map induces an isomorphism on the relevant homotopy groups.
Thus, all in all, the top arrow is zero for  $0 \leq i \leq 2p-3$. 
As the left hand arrow is an isomorphism, the bottom arrow has to be zero to make the diagram commute. 
Iterating this step implies that $X \wedge^L \theta=0$ for all $\theta$ in the non-negative degrees of $[M,\mathbb{S}]^{L_1\Sp}$. Altogether, we have that $X \wedge^L -$ acts trivially on the mod $p$ parts of $\pi_i(L_1 \mathbb{S})$, $i \neq 0$, and also the rational parts of $\pi_i(L_1 \mathbb{S})$, $i \neq 0$, which are trivial. Therefore, $X \wedge^L \theta$ is zero for all $\theta \in \pi_i(L_1\mathbb{S})$, $i\neq 0$ as desired.
\end{proof}

\begin{rem}
A similar argument shows that this is also true for the negative degree elements; indeed the bottom map is zero for $i+2p-2\geq 0$, so the top map is also the zero map in this range and iterating this gives the claim. It is not necessary for the proof of Theorem \ref{thm:klocalcohen} as we only require the claim for positive degrees, but we leave this here in case of a trivial $\pi_*(L_1\mathbb{S})$-action being required for future results. 
\end{rem}
 
 Let us now turn to another source of examples of trivial action. 
 \begin{defi}\label{nilpotence}
A set of ring spectra $\{R_i\}_{i \in I}$ \emph{detects nilpotence} if a map $f\colon \Sigma^k A \longrightarrow A$ in $\Ho(\Sp)$ for a finite spectrum $A$ is nilpotent if and only if $R_i \wedge^L f$ is zero for all $i$. 
 \end{defi}
\begin{prop}\label{lem:nilpotence}
Let $\{R_i\}_{i \in I}$ be a set of ring spectra which detects nilpotence. Then $\Ho(R_i\Mod)$ has trivial $\Ho(\Sp)$-action for all $i$.
\end{prop}
\begin{proof}
By \cref{prop:generators}, $\Ho(R_i\Mod)$ has trivial $\Ho(\Sp)$-action if and only if $R_i \wedge^L \theta=0$ for all $\theta \in \pi_i(\mathbb{S})$, $i>0$. Nishida's nilpotence theorem~\cite{nishida} shows that all positive degree elements of $\pi_*(\mathbb{S})$ are nilpotent and the claim follows.
\end{proof}

We will use the previous lemma to discuss examples relating to Morava $K$-theory and complex cobordism in the next section.

\changelocaltocdepth{2}
\section{Examples and applications}\label{sec:examples}
In this section we explore trivial actions and algebraicity in various examples arising in stable homotopy theory. Firstly, we explore the case of $E(1)$-local spectra together with exotic models for it, and then we turn to other examples arising from chromatic homotopy theory.

Most of the examples we are dealing with are $p$-local stable model categories in the sense that the action of $\Ho(\Sp)$-action factors over the $p$-local stable homotopy category $\Ho(\Sp_{(p)})$. In this case, one sees that $\C$ has a trivial $\Ho(\Sp)$-action if and only if the positive degree elements of $\pi_*(\mathbb{S}_{(p)})$ act trivially. Similarly, $\C$ has an essentially trivial $\Ho(\Sp)$-action if and only if the element $\alpha_1 \in \pi_{2p-3}(\mathbb{S})$ acts trivially ($p>2$), or if the Hopf elements $\eta$, $\nu$ and $\sigma$ act trivially ($p=2$). 

\subsection{$E(1)$-local spectra and exotic models}\label{sec:exoticE(1)}
In this section we study our algebraicity results in the case of $E(1)$-local spectra $L_1\Sp$ and exotic models for it. Recall that a stable model category $\C$ is said to be an \emph{exotic model} for $L_1\Sp$ if it is not Quillen equivalent to $L_1\Sp$ but does carry a triangulated equivalence as in the proposition below. 
\begin{prop}\label{cor:klocal}
Let $\Phi\colon \Ho(L_1\Sp) \xrightarrow{\sim} \Ho(\C)$ be a triangulated equivalence that is not derived from a Quillen functor, and let $p>2$. Then $\Ho(\Sp)$ acts essentially trivially on $\Ho(\C)$.
\end{prop}
\begin{proof}
By \cite[Proposition 6.7]{roitzheim2007rigidity}, $X \wedge \alpha_1=0$ for $X=\Phi(L_1S^0)$. As $X$ is a compact generator of $\Ho(\C)$, the claim follows from \cref{prop:generators}.
\end{proof}
\begin{rem}
One might imagine that \cref{thm:klocalcohen} can be applied in the previous corollary to deduce that $\Ho(\Sp)$ acts trivially on $\Ho(\C)$. However, \cite{roitzheim2007rigidity} shows that the mapping spectra in any exotic model are \emph{not} $E(1)$-local.
\end{rem}

Note that the existence of such a triangulated equivalence $\Phi$ for odd primes has been shown by~\cite{franke1996uniqueness, patchkoria2017exotic, patchkoriapstragowski} (the latter two papers correcting some gaps in the former). Next we recall the definition of Franke's exotic model for $E(1)$-local spectra; we refer the reader to~\cite{franke1996uniqueness, roi08, barnes_roitzheim_monoidal} for more details. We then turn to examining the exotic model in more detail with regards to algebraicity.

Given a Grothendieck abelian category $\A$, a self-equivalence $T\colon \A \to \A$ and an integer $N$, a \emph{twisted chain complex (of period $N$)} is a chain complex $X \in \mathrm{Ch}(\A)$ together with an isomorphism $\alpha_X \colon TX \to \Sigma^NX$. A map of twisted chain complexes $f\colon X \to Y$ is a chain map which is compatible with the specified isomorphisms $\alpha_X$ and $\alpha_Y$. We denote this category by $\tw$. There is an adjunction
\begin{equation}\label{eq:Frankeadjunction}
\begin{tikzcd}
\mathrm{Ch}(\A) \arrow[r, yshift=1mm, "\mathbb{P}"] & \tw . \arrow[l, yshift=-1mm, "\mathbb{U}"] 
\end{tikzcd}
\end{equation}
It is also helpful to note that if $\A$ is monoidal with unit object $\mathds{1}$, then (under some mild hypotheses) the category $\tw$ is isomorphic to the category of $\mathbb{P}\mathds{1}$-modules in $\mathrm{Ch}(\A)$.

Fix an odd prime $p$. Franke's exotic model for $E(1)$-local spectra is the category of twisted chain complexes of period $1$ over the abelian category $\A = E(1)_*E(1)\text{-}\mathrm{comod}$, of comodules over the flat Hopf algebroid $E(1)_*E(1)$. We denote this category by $\mathrm{Fr}_{1,p}$. This category admits several model structures, see~\cite{barnes_roitzheim_monoidal} for details. Firstly, $\mathrm{Fr}_{1,p}$ admits an injective model structure and this presents the homotopy theory of interest. As such, we will refer to $\mathrm{Fr}_{1,p}$ with the injective model structure as ``Franke's category''. However, the injective model structure is not monoidal in general (nor $\Ch$-enriched) and so for our purposes we are interested in the quasi-projective model structure, which is Quillen equivalent to the injective model structure as well as a monoidal model category. We briefly recall the construction of the quasi-projective model structure as in~\cite[\S 6]{barnes_roitzheim_monoidal} now. 

The category $\mathrm{Ch}(E(1)_*E(1)\text{-}\mathrm{comod})$ admits a relative projective model structure in the sense of~\cite{christensen_hovey} by taking the projective class generated by the dualizable $E(1)_*E(1)$-comodules, see also~\cite{hoveycomod} for more details. This can be right-lifted along the adjunction~(\ref{eq:Frankeadjunction}) to produce the relative projective model structure on $\mathrm{Fr}_{1,p}$. This model structure has fewer weak equivalences than the quasi-isomorphisms, so it does not present Franke's exotic model. Instead, left Bousfield localizing the relative projective model structure on $\mathrm{Fr}_{1,p}$ produces the quasi-projective model structure, in which the weak equivalences are the quasi-isomorphisms.
\begin{prop}\label{prop:fralg}
Franke's category $\mathrm{Fr}_{1,p}$ is an algebraic model category.
\end{prop}
\begin{proof}
Firstly, the underlying category of $\mathrm{Fr}_{1,p}$ is locally presentable by~\cite[Theorem 5.5.9]{Borceux} as $\mathrm{Ch}(E(1)_*E(1)\text{-}\mathrm{comod})$ is locally presentable. The quasi-projective model structure on $\mathrm{Fr}_{1,p}$ is cofibrantly generated by construction, and hence is combinatorial. We write $\underline{\Hom}(-,-)$ for the internal hom in $\mathrm{Fr}_{1,p}$ and note that this also provides the enrichment in $\Ch$ via the forgetful functor. Now, let $i\colon A \hookrightarrow X$ be a cofibration and $p\colon E \twoheadrightarrow B$ be a fibration in the quasi-projective model structure on $\mathrm{Fr}_{1,p}$. We must show that 
\[\underline{\Hom}(i^*, p_*) \colon \underline{\Hom}(X,E) \to \underline{\Hom}(X,B) \times_{\underline{\Hom}(A,B)} \underline{\Hom}(A,E)\] is a fibration in $\Ch$ which is moreover acyclic if either $i$ or $p$ is acyclic.

Since the quasi-projective model structure on $\mathrm{Fr}_{1,p}$ is monoidal by~\cite[Corollary 6.7]{barnes_roitzheim_monoidal}, we have that $\underline{\Hom}(i^*, p_*)$ is a quasi-projective fibration which is acyclic if either $i$ or $p$ is. By definition, $\underline{\Hom}(i^*, p_*)$ is also a quasi-projective fibration in $\mathrm{Ch}(E(1)_*E(1)\text{-}\mathrm{comod})$ which is acyclic if either $i$ or $p$ is. It follows that it is also a fibration in the relative projective model structure which is acyclic if either $i$ or $p$ is, since the quasi-projective model is a left Bousfield localization of the relative projective model. Any fibration (resp., weak equivalence) in the relative projective model is a surjection (resp., quasi-isomorphism) by~\cite[Proposition 2.1.5]{hoveycomod}. Therefore $\underline{\Hom}(i^*, p_*)$ is a surjection which is a quasi-isomorphism if either $i$ or $p$ is acyclic, as required.

We have now shown that the quasi-projective model structure on $\mathrm{Fr}_{1,p}$ is combinatorial, $\Ch$-enriched, and is Quillen equivalent to the injective model structure. Therefore the injective model structure on $\mathrm{Fr}_{1,p}$ is indeed algebraic.
\end{proof}

We now summarize the relationships between the various versions of algebraicity studied above for $E(1)$-local spectra and Franke's category.
\begin{exam}[$E(1)$-local spectra]\label{ex:L1Sp}
Consider the category of $E(1)$-local spectra at a fixed prime $p \geq 3$. Since there is a triangulated equivalence $\Ho(L_1\Sp) \simeq \Ho(\mathrm{Fr}_{1,p})$ one sees that $L_1\Sp$ is triangulated algebraic. As $$H\mathbb{F}_p \wedge L_1\mathbb{S} \simeq L_1(H\mathbb{F}_p) \simeq 0$$ (see~\cite[\S 5.2]{Gutierrez}), $L_1\mathbb{S}$ is not an $H\bbZ$-module. If it were, then by \cref{prop:CG}
\[
L_1 \mathbb{S} \simeq \bigvee_{i \in \mathbb{Z}} \Sigma^i H\pi_i(L_1\mathbb{S}),
\]
and therefore smashing with $H\mathbb{F}_p$ could not be trivial as for example $\pi_{2p-3}(L_1 \mathbb{S})=\mathbb{Z}/p$. 
 It then follows from \cref{cor:Rmodalgebraic} and \cref{prop:dz} that $L_1\Sp$ is neither algebraic nor has $\mathsf{D}(\bbZ)$-action. Finally, $\Ho(\Sp)$ does not act trivially on $\Ho(L_1\Sp)$ as $\alpha_1$ acts nontrivially.
\end{exam}

\begin{exam}[Franke's exotic model $\mathrm{Fr}_{1,p}$]\label{ex:fr}
Consider Franke's exotic model $\mathrm{Fr}_{1,p}$ for $E(1)$-local spectra at an odd prime. This is algebraic by \cref{prop:fralg}, and hence is triangulated algebraic, has $\mathsf{D}(\bbZ)$-action and has trivial $\Ho(\Sp)$ action by \cref{prop:alg_triang}, \cref{cor:dz} and \cref{cor:dz_trivial} respectively.
\end{exam}

\begin{exam}[General exotic models for $E(1)$-local spectra]\label{ex:E1exotic}
Finally we consider the case of a general exotic model for $L_1\Sp$ at an odd prime; that is, we consider a stable model category $\C$ which has a triangulated equivalence $\Phi\colon \Ho(L_1\Sp) \to \Ho(\C)$ that is not derived from a Quillen functor. The existence of $\Phi$ shows that $\C$ is triangulated equivalent to Franke's category $\mathrm{Fr}_{1,p}$ and hence $\C$ is triangulated algebraic. \cref{cor:klocal} shows that $\C$ also has essentially trivial $\Ho(\Sp)$-action. Unlike the case of Franke's exotic model described in \cref{ex:fr}, we do not know whether an arbitrary exotic model of $L_1\Sp$ is algebraic or has a $\mathsf{D}(\bbZ)$-action.
\end{exam}

\subsection{Morava $K$-theory}\label{sec:kn}

In this example we consider the category of $K(n)$-modules. We work at some implicit prime $p$.

Let $K(n)$ denote the Morava $K$-theory spectrum for $0 < n < \infty$. We see that $K(n)$ is not an $H\bbZ$-module since $H\mathbb{F}_p \wedge K(n) \simeq 0$~\cite[Theorem 2.1(i)]{ravenel1984localization}; if $K(n)$ were an $H\mathbb{Z}$-module spectrum, then its underlying spectrum would be a wedge of suspensions of $H\mathbb{F}_p$ by \cref{prop:CG}, and in particular $H\mathbb{F}_p \wedge K(n)$ would be nontrivial. 

Since $\{K(n)\}_{0 \leq n \leq \infty}$ detects nilpotence by~\cite[Theorem 3]{hopkins_smith}, the homotopy category of modules over the $n^{\mathrm{th}}$ Morava $K$-theory $K(n)$ has a trivial action of the stable homotopy category by \cref{lem:nilpotence}. However, we just noted (and will calculate explicitly for $n=1$ later in this subsection) that for $0 < n<\infty$ that $K(n)$ is not an $H\bbZ$-module. Therefore by \cref{cor:Rmodalgebraic} and \cref{prop:dz}, we see that $K(n)\Mod$ is neither algebraic, nor has a $\mathsf{D}(\bbZ)$-action. This illustrates that being algebraic and having $\mathsf{D}(\bbZ)$-action are both stronger properties than having trivial $\Ho(\Sp)$-action. 

We also know that 
\[
\pi_*\colon \Ho(K(n)\Mod) \longrightarrow \mathsf{D}(K(n)_*)
\]
is a triangulated equivalence. This is because $K(n)$ is a ``field'' in the sense that all modules over $K(n)$ are equivalent to a wedge of suspensions of $K(n)$ itself \cite{ravenel_orange}, while on the other side the analogue holds for a differential graded module over $K(n)_*$. Of course, $K(n)_*\Mod$ is algebraic, whereas we just illustrated why $K(n)\Mod$ is not, which in particular means that they are not Quillen equivalent, as pointed out by~\cite[Remark 2.5]{schwedeshipleyuniqueness}. The triangulated equivalence $\Ho(K(n)\Mod) \simeq \mathsf{D}(K(n)_*)$ makes $K(n)\Mod$ triangulated algebraic. 

We can see the extent to which $K(1)\Mod$ fails at being algebraic in the following obstruction theory.

Given a spectrum $E$, we would like to examine if there is a map from an Eilenberg-Mac~Lane spectrum $A$ with the correct homotopy groups to $E$. There is always a map from the respective Moore spectrum to $E$. An Eilenberg-Mac~Lane spectrum is constructed from a Moore spectrum by attaching cells, and we can see whether an already existing map from a CW-spectrum to $E$ lifts over the CW-spectrum with new cells attached. Note that we are just considering maps of spectra, not ring maps.

In more detail, we would like to see if there is a map 
 $$ \Sigma^i H(\pi_i(E)) \longrightarrow E$$ 
 inducing a $\pi_i$-isomorphism. If there is such a map for all $i$, then it assembles to a weak equivalence
 \[
 \bigvee_{i \in \mathbb{Z}} \Sigma^i H(\pi_i(E)) \longrightarrow E.
 \]
Without loss of generality, let $i=0$ and define $A:= H(\pi_i(E))$. Then $A= \colim_k A^{(k)}$, where $A^{(1)}$ is the $\pi_0(E)$-Moore spectrum, and $A^{(k+1)}$ is obtained from $A^{(k)}$ via attaching $(k+1)$-cells to kill the $k$-homotopy groups of $A^{(k)}$. In other words, we have the exact triangle 
\[
\bigvee\limits_{\phi \in \pi_k(A^{(k)})} \Sigma^k\mathbb{S} \xrightarrow{\vee \phi} A^{(k)} \longrightarrow A^{(k+1)} \longrightarrow \bigvee \Sigma^{k+1}\mathbb{S}.
\]
We can always construct a map $A^{(1)} \longrightarrow E$, via the following.

Let $G=\pi_0(E)$. Then $A^{(1)}$ is the Moore spectrum $MG$, constructed via an exact triangle
\[
\bigvee \mathbb{S} \xrightarrow{\rho} \bigvee\mathbb{S} \longrightarrow A^{(1)}=MG \longrightarrow \bigvee \Sigma\mathbb{S},
\]
where the first wedge of sphere runs over the relations in $G$ and the second wedge over the generators, see \cite[Example 7.4.7]{brbook}.
We extend this to a diagram
\[
\xymatrix{ \bigvee \mathbb{S} \ar[r]^{\rho} & \bigvee \mathbb{S} \ar[r] \ar[d]_{f_0=(\vee g)} & A^{(1)} \ar[r] \ar@{.>}[dl] & \bigvee \Sigma \mathbb{S}, \\
& E & & 
}
\]
where $g$ runs over the generators of $\pi_0(E)=[\mathbb{S},E]$. By construction, $f_0 \circ \rho =  \vee g \circ \rho =0$. Therefore, the dotted arrow exists, and we have a map $f_1\colon A^{(1)} \longrightarrow E.$

If we have a map $f_{k}\colon A^{(k)}\longrightarrow E$, then this extends to a map $f_{k+1}\colon A^{(k+1)}\longrightarrow E$ if and only if the composite
\[
\bigvee \Sigma^k\mathbb{S} \xrightarrow{\varphi=\vee\phi} A^{(k)} \xrightarrow{f_{k}} E
\]
is zero. Therefore, the obstruction for the existence of $f_k$ lies in 
\[
\varphi^*[A^{(k)},E] = \{ f \circ \varphi \mid f \in [A^{(k)},E] \} .
\]
In particular, if all those cosets are trivial, then we have the desired map $f\colon A \longrightarrow E$.

Now let us examine $\varphi^*[A^{(k)},E]$ in the case of $E=R\Hom(X,X)$, where $X$ is a compact generator of the stable model category $\C$. The map $f_{k} \circ \varphi$ is adjoint  to the map $F \circ (X \wedge^L \varphi)$, where $$F= \epsilon \circ (X \wedge^L f_{k}) \in [X \wedge^L A^{(k)}, X]^\C$$ is adjoint to $f_{k}$ and $$\epsilon\colon X \wedge^L R\Hom(X,X) \longrightarrow X$$ is the counit of the adjunction. Therefore we can say that the obstruction for the existence of $$f_{k+1}\colon A^{(k+1)} \longrightarrow R\Hom(X,X)$$ lies in
\[
(X \wedge^L \varphi)^* \Big( [X \wedge^L A^{(k)}, X]^\C\Big).
\]
For $R\Hom(X,Y)$, the obstructions lie in 
\[
(X \wedge^L \varphi)^* \Big( [X \wedge^L A^{(k)}, Y]^\C\Big).
\]
Therefore, by \cref{cor:emlaction} we have the following.

\begin{prop}
If there is a $\mathsf{D}(\bbZ)$-action on $\Ho(\C)$, then $$(X \wedge^L \varphi)^* \Big( [X \wedge^L A^{(k)}, Y]^\C\Big)=0$$ for all $k$. \qed
\end{prop}

\begin{exam}
Let us return to our example of Morava $K$-theory and $\C=K(1)\mbox{-mod}$. As noted at the beginning of the section, we are working $p$-locally.
We would like to use the obstruction theory to construct a map
\[
A=H\mathbb{Z}/p \longrightarrow E = \Hom_{K(1)}(K(1), K(1)) \cong K(1).
\]
As $K(1)$ is not an Eilenberg-Mac Lane spectrum, we will see exactly where this fails.
Note that $K(1)_*=\mathbb{F}_p[v_1, v_1^{-1}]$ with $|v_1|=2p-2$, so in particular $\pi_0(K(1))=\mathbb{Z}/p$, which is generated by the unit $\iota\colon \mathbb{S} \longrightarrow K(1).$ By the method described earlier, we have the following commutative diagram, in which $M$ is the mod-$p$ Moore spectrum.
\[
\begin{tikzcd} \mathbb{S} \arrow[r, "p"] & \mathbb{S} \arrow[rr, "\mathrm{incl}"] \arrow[d, "\iota"] &  & M=A^{(1)} \arrow[rr, "\mathrm{pinch}"] \arrow[dll, "\bar{\iota}"] & & \Sigma \mathbb{S} \\
 &  K(1) & &  & &
\end{tikzcd}
\]

The long exact sequence of homotopy groups tells us that the next nontrivial homotopy group of $M$ is
\[
\mathrm{incl}_*\colon \pi_{2p-3}(\mathbb{S}) \xrightarrow{\cong} \pi_{2p-3}(M),
\]
so $\pi_{2p-3}(M) \cong \mathbb{Z}/p$ is generated by $\mathrm{incl} \circ \alpha_1$.
Furthermore, with regards to the obstruction theory, we have that $$M=A^{(1)} = A^{(2)} = \cdots = A^{(2p-3)}.$$ To work out the next obstruction, we look at the following diagram.
\[
\begin{tikzcd}
\Sigma^{2p-3} \mathbb{S} \arrow[rr, "\mathrm{incl} \circ \alpha_1"] & &  A^{(2p-3)}=M \arrow[rr, "F"] \arrow[d, "\bar{\iota}"] & & A^{(2p-2)} \arrow[r, "G"] \arrow[dll, dotted, "\tilde{\iota}"] & \Sigma^{2p-2}\mathbb{S} \\
& & K(1) & & & 
\end{tikzcd}
\]
As we have that $\pi_{2p-3}(K(1))=0$, the map $\bar{\iota} \circ \mathrm{incl} \circ \alpha_1$ is zero for degree reasons, and the dotted map $\tilde{\iota}$ exists. The next step of the obstruction is to investigate the existence of the dotted arrow in this diagram.
\[
\begin{tikzcd} \bigvee \Sigma^{2p-2} \mathbb{S} \arrow[rr, "\vee\phi"] & &  A^{(2p-2)} \arrow[rr] \ar[d, "\tilde{\iota}"] & & A^{(2p-1)} \arrow[r] \arrow[dll, dotted] & \bigvee \Sigma^{2p-1}\mathbb{S} \\
& & K(1) & & & 
\end{tikzcd}
\]
To find out what $\tilde{\iota} \circ (\vee \phi)$ is, we need to know $\pi_{2p-2}(A^{(2p-2)})$. We apply the long exact sequence of homotopy groups to the exact triangle defining $A^{(2p-2)}$ and obtain
\[
\pi_{2p-2}(\Sigma^{2p-3} \mathbb{S}) \xrightarrow{(\mathrm{incl}\circ \alpha_1)_*} \pi_{2p-2}(M) \xrightarrow{F_*} \pi_{2p-2}(A^{(2p-2)}) \xrightarrow{G_*} \pi_{2p-2}(\Sigma^{2p-2} \mathbb{S}) \xrightarrow{(\mathrm{incl}\circ \alpha_1)_*} \pi_{2p-3}(M).
\]

Clearly, the first term of this sequence is trivial, so $F_*$ is an injection.

The long exact sequence of homotopy groups tells us that $\pi_{2p-3}(M) \cong \mathbb{Z}/p$ is generated by $\mathrm{incl} \circ \alpha_1$, so the final map of the previous sequence has kernel $p\mathbb{Z}$, which is also the image of $G_*$.



Again, we use the long exact sequence of homotopy groups to find out that 
\[
\mathrm{pinch}_*: \pi_{2p-2}(M) \longrightarrow \pi_{2p-3}(\mathbb{S})
\]
is an isomorphism, and we know that $ \pi_{2p-3}(\mathbb{S}) \cong \bbZ/p$ is generated by the element $\alpha_1$. 
Therefore,
$\pi_{2p-2}(M)\cong \mathbb{Z}/p$ is generated by an element $x$ with $\mathrm{pinch} \circ x= \alpha_1$. As 
\[
\alpha_1=\mathrm{pinch} \circ v_1 \circ \mathrm{incl}
\]
by \cite[Section 6.2]{roitzheim2007rigidity}, we have that $\pi_{2p-2}(M)$ is generated by $v_1 \circ \mathrm{incl}$.


As $F_*$ is an injection, we know that $F \circ v_1 \circ \mathrm{incl} \neq 0.$
Thus, we need to calculate $\tilde{\iota} \circ F \circ v_1 \circ \mathrm{incl}$. By construction, $\tilde{\iota} \circ F = \bar{\iota}.$ By \cite[Lemma 6.1.4]{ravenel_orange}, we have that $\bar{\iota}\circ v_1 = v_1 \circ \bar{\iota}$, so \[
\tilde{\iota} \circ F \circ v_1 \circ \mathrm{incl} = \bar{\iota} \circ v_1 \circ \mathrm{incl} = v_1 \circ \bar{\iota} \circ \mathrm{incl} = v_1 \circ \iota.
\]
This is precisely the generator of $\pi_{2p-2}(K(1))$ and hence not trivial. Therefore, the obstruction to extend $\tilde{\iota}$ to $A^{(2p-1)}$ is nonzero, and we arrive at the expected result that there is no nontrivial map from $H\mathbb{Z}/p$ to $K(1)$, thus revisiting the result that $K(1)$ is not an Eilenberg-Mac Lane spectrum. In particular, the category of $K(1)$-modules does not carry a $\mathsf{D}(\mathbb{Z})$-action. An analogous calculation holds for $k(1)$ instead of $K(1)$.

It is a subject for future research to see if the obstruction theory can be exploited further to determine the level of algebraicity of some other examples. 
\end{exam}

We now summarize the key findings of this section in the following example.
\begin{exam}\label{exam:moravaK}
For any $0 < n < \infty$, the Morava $K$-theory spectrum $K(n)$ is not an $H\bbZ$-module and therefore the category $K(n)\Mod$ is neither algebraic nor has $\mathsf{D}(\bbZ)$-action. However, there is a triangulated equivalence $\Ho(K(n)\Mod) \simeq \mathsf{D}(K(n)_*)$ making $K(n)\Mod$ triangulated algebraic. The category $K(n)\Mod$ also has trivial $\Ho(\Sp)$-action. On the other hand, the category $K(n)_*\Mod$ is algebraic. As such, the triangulated equivalence on homotopy categories between $K(n)\Mod$ and $K(n)_*\Mod$ cannot be lifted to a Quillen equivalence.
\end{exam}

\subsection{Further examples from chromatic homotopy theory}
In this subsection, we explore some other examples, which arise from the work of~\cite{patchkoria2012alg, patchkoria2017derived, patchkoria2017exotic, patchkoriapstragowski}.

In general, it is difficult to establish whether a model category is triangulated algebraic, if it is not already algebraic itself. In other words, we are looking for model categories with an algebraic ``exotic model'', like those from the Franke-style machinery. In the rest of this subsection we explore some more triangulated algebraic examples arising in chromatic homotopy theory.
\begin{exam}\label{ex:irakli}
From~\cite[Corollary 8.3]{patchkoriapstragowski} there are triangulated equivalences between each of the following model categories and an algebraic model category. Note that these are all corollaries from the same general theorem, which gives a triangulated equivalence between $\Ho(R\Mod)$ and $\mathsf{D}(\pi_*(R))$ under certain conditions on $\pi_*(R)$ related to sparseness and global dimension. Since these have different behaviour with regards to the action of $\Ho(\Sp)$, we split up these examples into three types.
\begin{enumerate}
\item $E(1)$-local examples:
\begin{itemize}
\item $E(1)$-modules for $p>2$, $\pi_*(E(1)) = \bbZ_{(p)}[v_1,v_1^{-1}]$, $|v_1|=2p-2$;
\item $KO_{(p)}$-modules for $p>2$, $\pi_*(KO_{(p)})= \mathbb{Z}_{(p)}[v,v^{-1}]$, $|v|=4$;
\item $KU_{(p)}$-modules for $p>2$, $\pi_*(KU_{(p)})=\mathbb{Z}_{(p)}[\beta, \beta^{-1}]$, $|\beta|=2$.
\end{itemize}
For a ring spectrum $R$, the homotopy mapping spectra of $R$-modules are of course $R$-modules themselves and as such, they are also $R$-local spectra \cite[Proposition 1.17(a)]{ravenel1984localization}. A spectrum is $KU_{(p)}$-local if and only if it is $KO_{(p)}$-local if and only if it is $E(1)$-local~\cite[Theorem 8.14]{ravenel1984localization}. As such in each of these examples the homotopy mapping spectra are $E(1)$-local, and hence the category has trivial $\Ho(\Sp)$-action by \cref{thm:klocalcohen} since $\alpha_1$ acts trivially for degree reasons.
\item Examples via nilpotence:
\begin{itemize}
\item modules over connective Morava-$K$-theory $k(n)$ for $0 < n < \infty$ with $p^n > 2$, where $k(n)_*=\mathbb{F}_p[v_n]$, $|v_n|=2p^n-2$.
\end{itemize}
We claim that the set $\{k(n)\}_{0 \leq n \leq \infty}$ detects nilpotence and therefore each $k(n)\Mod$ has trivial $\Ho(\Sp)$-action by \cref{lem:nilpotence}. In order to show this, we will prove that a set $\{R_i\}$ of ring spectra detects nilpotence if for all $0 \leq n \leq \infty$, there exists an $i$ such that $K(n)_*(R_i) \neq 0$. The argument for this is a slight modification of that of~\cite[Proof of Corollary 5]{hopkins_smith}. Let $f\colon \Sigma^k A \to A$ be a map of finite spectra. Since $K(n)$ is a field, $K(n) \wedge R_i$ is a wedge of suspensions of $K(n)$, and therefore if $R_i \wedge^L f = 0$ for all $i$ it follows that $K(n) \wedge^L f = 0$ for all $n$. Since $\{K(n)\}_{0 \leq n \leq \infty}$ detects nilpotence by~\cite[Theorem 3]{hopkins_smith} it follows that $f$ is nilpotent, and hence that $\{R_i\}$ detects nilpotence. Since $k(n) \wedge K(n) \not\simeq 0$~\cite[Theorem 2.1(e)]{ravenel1984localization}, the previous criterion coupled with \cref{lem:nilpotence} shows that $k(n)\Mod$ has trivial action.
\item Other examples: 
\begin{itemize}
\item  $E(n)$-modules for $2p-n>3$, where $\pi_*(E(n)) = \bbZ_{(p)}[v_1,\ldots ,v_{n-1}, v_n, v_n^{-1}]$, $|v_i| = 2p^i-2$;
\item modules over the truncated Brown-Peterson spectrum $BP\langle n\rangle$ for $2p-n>4$, \newline $BP\langle n\rangle_*= \mathbb{Z}_{(p)}[v_1, \cdots , v_n]$, $|v_i|=2p^i-2$.
\end{itemize}
By \cref{ringcoefficients}, these examples have trivial $\Ho(\Sp)$-action. 
\end{enumerate}
All of the above examples are triangulated algebraic. However, in each case $R$ is not a generalized Eilenberg-Mac Lane spectrum (see~\cite[Appendix A]{patchkoria2012alg} for details), and hence by \cref{cor:Rmodalgebraic} and \cref{prop:dz}, none of the above are algebraic nor have a $\mathsf{D}(\bbZ)$-action.
\end{exam}

\subsection{Complex cobordism}
In this section, we consider the example of the complex cobordism spectrum $MU$, and investigate what kinds of algebraicity $MU\Mod$ enjoys.
\begin{exam}\label{MU}
Since $\{MU\}$ detects nilpotence by~\cite[Theorem 1]{dhs} it follows from \cref{lem:nilpotence} that $\Ho(MU\Mod)$ has trivial $\Ho(\Sp)$-action. (Alternatively, this also follows from \cref{ringcoefficients}.) One sees that $MU\Mod$ is not algebraic by \cref{cor:Rmodalgebraic}; if $MU$ were an $H\bbZ$-algebra then any $MU$-module would be an $H\bbZ$-module by restriction of scalars but this fails for example for $K(n)$, see \cref{exam:moravaK}. One can also see that $MU\Mod$ does not have a $\mathsf{D}(\bbZ)$-action as follows. The homotopy mapping spectrum $\Hom_{MU}(MU, K(n)) \simeq K(n)$ is not an $H\bbZ$-module, and so by \cref{prop:dz} the action of $\Ho(\Sp)$ does not factor over $\mathsf{D}(\bbZ)$.

To the best of the authors' knowledge, whether or not $MU\Mod$ is triangulated algebraic is an open question. Schwede has introduced the $n$-order of a triangulated category as an invariant; the $n$-order of a triangulated algebraic category is infinite~\cite[Theorem 2.1]{schwedealg}, but the converse need not hold. The category $\Ho(MU\Mod)$ has infinite $n$-order~\cite[Example 1.8]{schwedealg}, so this invariant is not conclusive in this example.
\end{exam}

\subsection{Endomorphisms of $H\bbZ$ over $ku$}\label{sec:kuHZ}
In this section we explore an interesting example which shows (amongst other things) that having a $\mathsf{D}(\bbZ)$-action is weaker than being algebraic. 

Let $ku$ denote the connective complex $K$-theory spectrum, and recall that $ku_* = \bbZ[\beta]$ where $\beta$ is the Bott element in degree $2$.  By killing homotopy groups, there is a ring map $ku \to H\bbZ$. Note moreover that $H\bbZ$ is a compact $ku$-module since it is equivalent (as a $ku$-module) to the Koszul spectrum
\[\kos{ku}{\beta} = \mathrm{cofib}(\Sigma^2 ku \xrightarrow{\beta} ku).\] In this example we consider the endomorphism ring spectrum $\E = \mathrm{End}_{ku}(H\bbZ)$. Before we can discuss what types of algebraicity the category of $\E$-modules enjoys we require some preparatory results.

There is a ring map $c\colon ku \to \mathrm{End}_\E(H\bbZ)$ adjoint to the action map of $ku$ on $H\bbZ$, called the \emph{double centralizer.}
\begin{lemma}\label{dccomplete}
The map $c\colon ku \to \mathrm{End}_\E(H\bbZ)$ is an equivalence; in the language of~\cite[\S 4.16]{DGI}, the map $ku \to H\bbZ$ is \emph{dc-complete}.
\end{lemma}
\begin{proof}
 The double centralizer may be identified with the $\beta$-completion $\Lambda_\beta(ku)$ of $ku$~\cite{DG02} (also see~\cite[\S 6.C]{Greenlees18}) so it suffices to show that $ku$ is $\beta$-complete. There is a spectral sequence~\cite[3.3]{GM95}
\[E^2_{s,t} = H_s^{\beta}(ku_*)_t \Longrightarrow \pi_{s+t}(\Lambda_\beta(ku))\]
where $H_*^{\beta}(ku_*)$ denotes the local homology groups of $ku_*$ at the ideal $(\beta)$ in the sense of~\cite{GM92}. Since the coefficient ring $\bbZ$ of $ku_* = \bbZ[\beta]$ is in degree 0 and $\beta$ is in degree 2, one sees that $\bbZ[\beta]$ agrees with the power series ring $\bbZ\llbracket\beta\rrbracket$ by comparing homogeneous parts. As such, $ku_* = \bbZ[\beta]$ is $\beta$-adically complete, and thus the local homology groups are \[H_s^{\beta}(ku_*) = \begin{cases} 
ku_* & s=0 \\
0 & \mathrm{otherwise} \end{cases}\]
by~\cite[Theorem 4.1]{GM92}. Therefore the spectral sequence collapses at the $E^2$-page to show that $ku$ is $\beta$-complete, and hence that $c$ is an equivalence.
\end{proof}

Using the previous lemma, we now show that whilst the endomorphism ring spectrum $\E$ is an $H\bbZ$-module, it is not an $H\bbZ$-algebra. 
\begin{thm}\label{notHZalg}
The endomorphism ring spectrum $\E = \Hom_{ku}(H\bbZ, H\bbZ)$ is an $H\bbZ$-module, but is not weakly equivalent to an $H\bbZ$-algebra as a ring spectrum.
\end{thm}
\begin{proof}
There is a ring map $H\bbZ \to \E$ which is adjoint to the map $H\bbZ \wedge_{ku} H\bbZ \to H\bbZ$. This gives $\E$ the structure of an $H\bbZ$-module by restriction.

We now suppose that $\E$ is an $H\bbZ$-algebra and produce a contradiction. If $\E$ were an $H\bbZ$-algebra, then $\mathrm{End}_\E(H\bbZ)$ would also be an $H\bbZ$-algebra. If this were the case, then we would have a commutative diagram
\[
\begin{tikzcd}
\s \arrow[r, "\iota"] \arrow[d] & ku \arrow[d, "c"] \\
H\bbZ \arrow[r] &  \mathrm{End}_\E(H\bbZ) 
\end{tikzcd}
\]
since $\s$ is initial. As $c$ is an equivalence by \cref{dccomplete}, this would mean that the unit map of $ku$ factors over $H\bbZ$, which is false as $[H\bbZ, ku]$ is zero in degree zero. This is well-known but for completeness we recall the argument in \cref{lem:hzku} below.
As such, $\E$ cannot be weakly equivalent to an $H\bbZ$-algebra as a ring spectrum. 
\end{proof}

We now turn to relating the above observations to algebraicity statements.
\begin{exam}\label{ex:ku}
Consider the $H\bbZ$-cellularization of $ku$-modules which we denote by \linebreak $\mathrm{Cell}_{H\bbZ}(ku\Mod)$. This is sometimes called the right Bousfield localization at $H\bbZ$ and denoted by $R_{H\bbZ}(ku\Mod)$. The homotopy category of this model category is the localizing subcategory $\mathrm{Loc}_{ku}(H\bbZ)$ of $ku$-modules generated by $H\bbZ$, and $H\bbZ$ is a compact generator for it~\cite[Corollary 2.6]{GreenleesShipley13}. By Morita theory we have a Quillen equivalence $$\mathrm{Cell}_{H\bbZ}(ku\Mod) \simeq_Q \rmod\E,$$ see~\cite[Theorem 2.1]{DG02} and~\cite[Theorem 8.7]{barnes_roitzheim_stable}. Since $\E$ is spectral with enrichment given by $\Hom_\E(-,-)$, we have $\mathrm{hEnd}(\E) = \E$. Since $\E$ is not weakly equivalent to an $H\bbZ$-algebra as a ring spectrum by \cref{notHZalg}, we see that $\mathrm{Cell}_{H\bbZ}(ku\Mod)$ and $\rmod\E$ are not algebraic  by \cref{prop:algebraicEMLHZalg}. 


We now show that $\rmod\E$ (and hence $\mathrm{Cell}_{H\bbZ}(ku\Mod)$) has a $\mathsf{D}(\bbZ)$-action. There exists a ring map $\theta\colon H\bbZ \to \E$ adjoint to the map $H\bbZ \wedge_{ku} H\bbZ \to H\bbZ.$ The category $\Ho(\rmod\E)$ is enriched, tensored and cotensored over itself, and hence is enriched, tensored and cotensored over $\mathsf{D}(\bbZ)$ via extension and restriction of scalars along the ring map $\theta$, see~\cite[3.7.11]{Riehl} for instance. Therefore $\rmod\E$ has a $\mathsf{D}(\bbZ)$-action. This shows that the converse to \cref{cor:dz} is false, because we have found a model category with $\mathsf{D}(\bbZ)$-action which is not algebraic. Furthermore, we can also see from \cref{cor:dz_trivial} that $\rmod\E$ has a trivial $\Ho(\Sp)$-action.  

We have also computed the homotopy groups of $\E$, but they do not satisfy the hypotheses of~\cite{patchkoria2012alg, patchkoria2017derived, patchkoria2017exotic, patchkoriapstragowski} regarding global dimension and sparseness, so we cannot draw any conclusions about whether or not $\rmod\E$ is triangulated algebraic at this stage.

We finish this example with a lemma we needed for \cref{notHZalg}.

\begin{lemma}\label{lem:hzku}
The group $[H\bbZ, ku]_0$ is trivial.
\end{lemma}

\begin{proof}
Firstly, we note that the connective cover functor is right adjoint to the inclusion of the full category of connective spectra into spectra. Thus, 
\[
[H\bbZ, ku]_* = [H\bbZ, KU]_*=KU^*(H\bbZ),
\]
which is of course 2-periodic. There is a short exact sequence
\[
0 \rightarrow \mathrm{Ext}^1_{\bbZ}(KU_{-1}(H\bbZ), \bbZ) \rightarrow KU^0 (H\bbZ) \rightarrow \Hom_{\bbZ}(KU_0(H\bbZ), \bbZ) \rightarrow 0,
\]
see \cite{HeardStojanoska14}, so we need to calculate $KU_*(H\bbZ)$. We note that $L_{KU}H\bbZ= H\mathbb{Q}$ \cite[Proposition 5.2]{Gutierrez}, so
\[
KU_*(H\bbZ)= KU_*(L_{KU}H\bbZ) = \pi_*(H\mathbb{Q} \wedge KU),
\]
which is $\mathbb{Q}$ in even degrees and zero in odd degrees. Putting this into the previous short exact sequence yields that $[H\bbZ, KU]_{\mathrm{even}}= \Hom_{\bbZ}(\mathbb{Q}, \bbZ)=0$  and $[H\bbZ, KU]_\mathrm{odd}= \mathrm{Ext}^1_{\bbZ}(\mathbb{Q}, \bbZ).$
\end{proof}

\end{exam}

\subsection{Summary of our findings}
In this subsection we summarize the main findings of this paper. \cref{fig:summary} shows a summary of how the different notations of algebraicity relate, and \cref{table:examples} recaps the properties of the various examples studied. \begin{figure}[H]
\centering
\begin{tikzpicture}
        \node [decision, yshift=5em] (Algebraic) {Algebraic};
    \node [block, below left of=Algebraic, xshift=-5em] (DZ) {$\mathsf{D}(\mathbb{Z})$-action};
    \node [block, below right of=Algebraic, xshift=5em] (tria) {Triangulated algebraic};
    \node[draw,
    decision,
    below=3cm of Algebraic,
    minimum width=2.5cm,
    minimum height=1cm,] (Trivial action) {Trivial action};
    \draw [vecArrow,red] (DZ) to[]+ (0,1)  |-  node[pos=0.25,fill=white,inner sep=1pt]{(\ref{ex:ku})} node[pos=0.75,fill=white,inner sep=2pt]{\cross} (Algebraic);
    \draw [vecArrow, red] (tria) to[]+ (0,1)  |-  node[pos=0.2,fill=white,inner sep=1pt]{(\ref{exam:moravaK})} node[pos=0.75,fill=white,inner sep=2pt]{\cross} (Algebraic);

   \draw [vecArrow] (Algebraic) -- (DZ)node[pos=0.5,fill=white,inner sep=1pt]{(\ref{cor:dz})};
      \draw [vecArrow] (Algebraic) -- (tria)node[pos=0.5,fill=white,inner sep=1pt]{(\ref{prop:alg_triang})};

        \draw [vecArrow] (DZ) -- (Trivial action)node[pos=0.45,fill=white,inner sep=1pt]{(\ref{cor:dz_trivial})};
   \draw [vecArrow,red] (tria) -- (Trivial action) node[pos=0.35,fill=white,inner sep=1pt]{(\ref{ex:L1Sp})} node[pos=0.65,fill=white,inner sep=2pt]{\cross};

    \draw [vecArrow, red] (Trivial action) -| node[pos=0.25,fill=white,inner sep=0]{(\ref{exam:moravaK})} node[pos=0.75,fill=white,inner sep=1pt]{\cross} (DZ);
   \draw [vecArrow, dashed]  (Trivial action) -| node[pos=0.25,fill=white,inner sep=1pt]{\textbf{?}} (tria);
      \draw [vecArrow,red] (Trivial action) -- (Algebraic) node[pos=0.25,fill=white,inner sep=1pt]{(\ref{exam:moravaK})} node[pos=0.5,fill=white,inner sep=4pt]{} node[pos=0.75,fill=white,inner sep=2pt]{\cross};
          \draw [vecArrow,red] (tria) -- (DZ) node[pos=0.25,fill=white,inner sep=0]{(\ref{exam:moravaK})} node[pos=0.75,fill=white,inner sep=2pt]{\cross};
\end{tikzpicture}
\caption{Summary of the relations between different notions of algebraicity. The implications which hold are in black, whereas those that fail are in red.}
\label{fig:summary}
\end{figure}
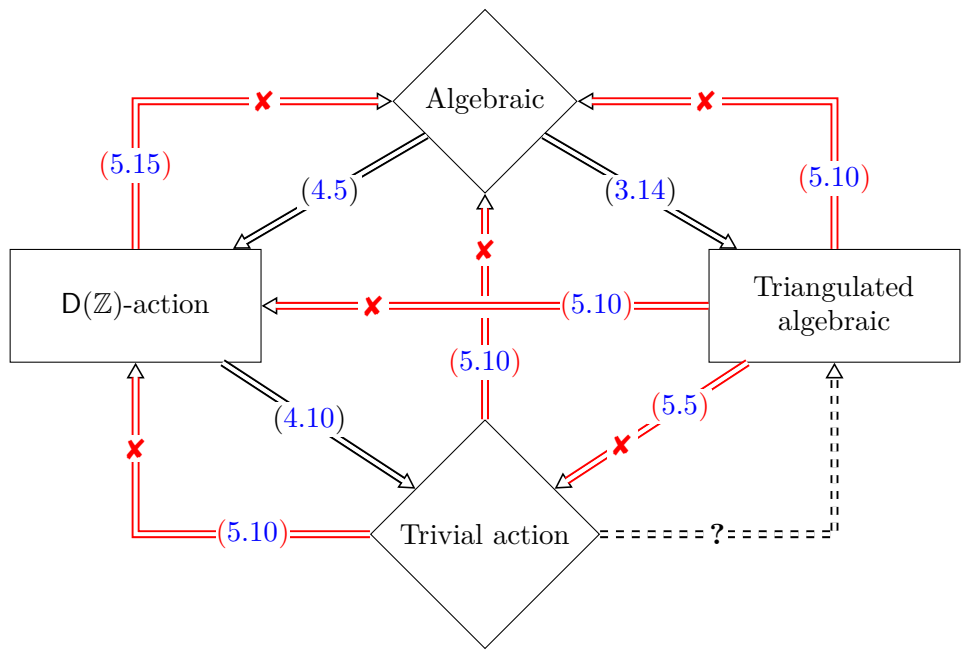

It is hard to imagine that the implication from trivial action to triangulated algebraic does hold, but the authors could not find a counterexample at the moment. The same applies to the implication from $\mathsf{D}(\mathbb{Z})$-action to triangulated algebraic. Of course, if the dotted implication in the diagram does hold, then $\mathsf{D}(\mathbb{Z})$-action would imply triangulated algebraic. 

\setlength{\tabcolsep}{0.2em}
\renewcommand{\arraystretch}{1.5}
\begin{table}[H]
\begin{tabular}{m{0.28\linewidth}|m{0.11\linewidth}|m{0.1\linewidth}|m{0.14\linewidth}|m{0.09\linewidth}|m{0.15\linewidth}|c}
\centering Category & \centering Algebraic & \centering $\mathsf{D}(\bbZ)$-action & \centering Triangulated algebraic & \centering Trivial action & \centering Essentially trivial action & Ref. \\
\hline
\centering $L_1\Sp$ ($p>2$)& \centering \cross & \centering \cross & \centering \tick & \centering \cross & \centering \cross & \ref{ex:L1Sp} \\
\hline
\centering Franke's exotic model \\ for $L_1\Sp$ ($p>2$) & \centering \tick & \centering \tick & \centering \tick & \centering \tick & \centering \tick &  \ref{ex:fr} \\
\hline

\centering General exotic models \\ for $L_1\Sp$ ($p>2$)& \centering \textbf{\textcircled{\tiny?}} & \centering \textbf{\textcircled{\tiny?}} & \centering \tick & \centering \textbf{\textcircled{\tiny?}} & \centering \tick &  \ref{ex:E1exotic} \\
\hline
\centering $K(n)_*\Mod$ & \centering \tick & \centering \tick & \centering \tick & \centering \tick & \centering \tick & \ref{exam:moravaK} \\
\hline
\centering $K(n)\Mod$ ($0 < n < \infty$) & \centering \cross & \centering \cross & \centering \tick & \centering \tick & \centering \tick & \ref{exam:moravaK} \\
\hline

\centering $E(1)\Mod$ ($p>2$) \\
$KO_{(p)}\Mod$ ($p>2$)  \\
$KU_{(p)}\Mod$ ($p>2$)  & \centering \cross & \centering \cross & \centering \tick & \centering \tick& \centering \tick & \ref{ex:irakli} \\
\hline
\centering $k(n)\Mod$ \\
($0 < n < \infty$ and $p^n>2$) & \centering \cross & \centering \cross & \centering \tick & \centering \tick& \centering \tick & \ref{ex:irakli} \\

\hline
\centering $E(n)\Mod$ ($2p-n>3$) \\
$BP\langle n\rangle\Mod$ ($2p-n > 4$) & \centering \cross & \centering \cross & \centering \tick & \centering \tick & \centering \tick & \ref{ex:irakli} \\
\hline

\centering $MU\Mod$ & \centering \cross & \centering \cross & \centering \textbf{\textcircled{\tiny?}} & \centering \tick & \centering \tick &  \ref{MU} \\

\hline
\centering $\mathrm{Cell}_{H\bbZ}(ku\Mod)$ & \centering \cross & \centering \tick & \centering \textbf{\textcircled{\tiny?}} & \centering \tick & \centering \tick & \ref{ex:ku} \\

\end{tabular}
\caption{Summary of examples}
\label{table:examples}
\end{table}

\end{document}